\documentclass[12pt]{amsart}
\usepackage{amsmath}
\usepackage{amsxtra}
\usepackage{amstext}
\usepackage{amssymb}
\usepackage{amsthm}
\usepackage{latexsym}
\usepackage{dsfont} % for \mathbf{N}
\usepackage{verbatim}
\usepackage{tabls}
\usepackage{graphicx}
\usepackage{rotating}
\usepackage{caption}
\theoremstyle{plain}
\newtheorem{thm}{Theorem}[section]
\newtheorem{cor}[thm]{Corollary}
\newtheorem{lem}[thm]{Lemma}
\newtheorem{prop}[thm]{Proposition}
\newtheorem{ques}[thm]{Question}

\theoremstyle{definition}
\newtheorem{defn}[thm]{Definition}

\newtheorem{cla}[thm]{Claim}

\numberwithin{equation}{section}
\def\R1{\widetilde{R}}
\def\T1{\widetilde{T}}
\def\dist{\operatorname{dist}}
\def\supp{\operatorname{supp}}
\def\Lip{\operatorname{Lip}}
\def\eps{\varepsilon}
\def\kap{\varkappa}

\def\diam{\operatorname{diam}}
\def\loc{\operatorname{loc}}
\def\e{\textbf{e}}

\def\XXint#1#2#3{{\setbox0=\hbox{$#1{#2#3}{\int}$}
     \vcenter{\hbox{$#2#3$}}\kern-.5\wd0}}

\begin{document}

\title[Reflectionless measures]
{Reflectionless measures for Calder\'{o}n-Zygmund operators.}

\author[B. Jaye]
{Benjamin Jaye}
\address{Department of Mathematical Sciences,
Kent State University,
Kent, OH 44240, USA}
\email{bjaye@kent.edu}

\author[F. Nazarov]
{Fedor Nazarov}
%\address{Department of Mathematics,
%Kent State University,
%Kent, OH 44240, USA}
\email{nazarov@math.kent.edu}

\date{\today}

\begin{abstract}  We study the properties of reflectionless measures for a Calder\'{o}n-Zygmund operator $T$.  Roughly speaking, these are measures $\mu$ for which $T(\mu)$ vanishes (in a weak sense) on the support of the measure. We describe the relationship between certain well-known problems in harmonic analysis and geometric measure theory and the classification of reflectionless measures.  As an application of our theory, we give a new proof of a recent theorem of Eiderman, Nazarov, and Volberg, which states that in $\mathbb{R}^d$, the $s$-dimensional Riesz transform of a non-trivial $s$-dimensional measure is unbounded if $s\in (d-1,d)$.\end{abstract}

\maketitle

\section{Introduction}

%In this paper we study the regularity properties that are inherited by a measure as a result of the boundedness of an associated Calder\'{o}n-Zygmund operator.

Fix an integer $d\geq 2$, and let $s\in (0,d)$.  For a measure $\mu$,  the $s$-dimensional Calder\'{o}n-Zygmund operator (CZO) $T_{\mu}$ is formally defined by $T_{\mu}(f)(x) = \int_{\mathbb{R}^d} K(x-y)f(y)d\mu(y)$, where $K$ is an $s$-dimensional CZ-kernel (that is, $K$ is odd, $K(\lambda x)=\lambda^{-s} K(x)$ for $\lambda \in (0,\infty)$, and $K$ is H\"{o}lder continuous on the unit sphere).

We are primarily interested in the connection between geometric properties of a measure, and regularity properties of a CZO, such as the following:

\begin{equation}\tag{$\dagger$}\begin{split}
   &\textit{Let }\mu \textit{ be an s-dimensional measure.\! If }T_{\mu} \!\textit{ is bounded in }L^2(\mu), \\
   &\textit{then }\mu\textit{ is the zero measure}.
\end{split}\end{equation}

A measure $\mu$ is said to be $s$-dimensional if there is a constant $\Lambda >0$ such that $\mu(B(x,r))\leq \Lambda r^s$ for any ball $B(x,r)\subset\mathbb{R}^d$, and $\mathcal{H}^s(\supp(\mu))<\infty$.  The precise meaning of the boundedness of  $T_{\mu}$ in $L^2(\mu)$ is a rather standard one, namely that regularized CZOs associated with $T_{\mu}$ are bounded uniformly in terms of the regularization parameter (see Definition \ref{gooddef} below).

Our interest in this problem comes from a well-known conjecture, which predicts that the $s$-Riesz transform ($K(x)= \tfrac{x}{|x|^{s+1}}$, $x\in \mathbb{R}^d$) has the property ($\dagger$) if $s\not\in \mathbb{Z}$.  (Simple examples show that integer dimensional Riesz transforms fail to satisfy this property.)   This conjecture has been proved by Prat \cite{Pra} in the case when $s\in (0,1)$, and by Eiderman, Nazarov and Volberg \cite{ENV} if $s\in (d-1,d)$.   The remaining cases remain open.

This conjecture remains one of the outstanding problems in the subfield of geometric measure theory that seeks to understand the regularity properties of non-integer dimensional sets.  For the remainder of this paragraph, fix $s\not\in \mathbb{Z}$.  In his fundamental work, Marstrand \cite{Mar} proved that there does not exist a non-trivial measure $\mu$ whose $s$-density exists $\mu$-almost everywhere (see Theorem 14.1 of \cite{Mat}).  More recently, the focus has shifted to attempting to understand the behaviour that one can expect from a singular integral on an arbitrary $s$-dimensional set. This appears to be more difficult  than problems involving the existence of density, or, rather, the techniques developed thus far are less adequate.  Ruiz de Villa and Tolsa \cite{RT} (in a sequel to \cite{Tol}) proved that, for an $s$-dimensional measure $\mu$, the $\mu$-almost everywhere existence of the $s$-Riesz transform of $\mu$ in the sense of principal value implies that $\mu$ is zero.  The property of $L^2(\mu)$ boundedness carries (a priori) significantly less information than the existence of principal values, and all the known results (except those in \cite{Pra, ENV}) require additional hypotheses.  Vihtila \cite{Vih} proved that if $\mu$ is an $s$-dimensional measure satisfying $\liminf_{r\rightarrow 0}\tfrac{\mu(B(x,r))}{r^s}>0$ for $\mu$-almost every $x\in \mathbb{R}^d$, then the $s$-Riesz transform of $\mu$ is unbounded. The lower density assumption here is essential to the argument. The behaviour of the $s$-Riesz transform of the Hausdorff measure restricted to a well separated, self similar, $s$-dimensional fractal set is also understood, see \cite{EV, Tol11}.

Our first result reduces the question of whether a CZO $T$ has property ($\dagger$) to describing the reflectionless measures associated with $T$.  These are (non-negaive, non-atomic, possibly infinite) measures $\mu$ for which $T_{\mu}(1)$ (considered in a suitable weak sense) vanishes on the support of $\mu$ (see Definition \ref{refldef}).  The study of reflectionless measures is of inherent interest in harmonic analysis, see for example \cite{MPV, TV}.

\begin{thm}\label{reducerefl}  If the only reflectionless measure for a CZO $T$ is the zero measure, then property $\operatorname{(}\dagger\operatorname{)}$ holds.
\end{thm}

It is possible to deduce certain statements in the opposite direction to Theorem \ref{reducerefl}.  For instance, if a CZO $T$ exhibits a non-trivial compactly supported reflectionless measure with zero lower s-density, then property ($\dagger$) fails.  A thorough development of this idea will be taken up elsewhere.

 After establishing Theorem \ref{reducerefl}, we move onto describing finer properties of reflectionless measures $\mu$ for general CZOs $T$.  Our main result here, called the Collapse Lemma (see Section \ref{colsec}), roughly states that $\mu$ is porous in balls where $T_{\mu}(1)$ is large.

This result is in a sense optimal for a general CZO.  Indeed, in the recent preprint \cite{JN2}, we observed that the two dimensional Lebesgue measure of a disc is a reflectionless measure for the CZO with kernel $K(z) = \tfrac{\overline{z}}{z^2}$, $z\in \mathbb{C}$.  This yields an example where $T_{\mu}(1)$  vanishes identically on a non-porous set.  We remark that this observation was used in \cite{JN2} to construct a $1$-dimensional purely unrectifiable measure $\mu$ with respect to which the CZO with kernel $\tfrac{\bar{z}}{z^2}$ is bounded in $L^2(\mu)$.

In the third part of the paper, we specialize to the case of the $s$-Riesz transform.   We pose the following question regarding the structure of reflectionless measures for the $s$-Riesz transform:

\begin{ques}\label{reflconj}  Suppose that $\mu$ is a reflectionless measure for the $s$-Riesz transform.  %Is it true that $\mu$ has the following structure?
\begin{enumerate}
\item If $s\not\in \mathbb{Z}$, then is $\mu$ necessarily the zero measure?
\item If $s\in \mathbb{Z}$, and $\mu$ is Ahlfors-David regular (meaning that there is a constant $c_0$ such that $\mu(B(x,r))\geq c_0 r^s$ for all $x\in \supp(\mu)$ and $r>0$), then must $\mu$ coincide with a constant multiple of the $s$-dimensional Hausdorff measure restricted to an $s$-plane?
\end{enumerate}
\end{ques}

An affirmative answer to part (2) of this question would yield the solution to a well-known question of David and Semmes on the uniform $s$-rectifiability of an Ahlfors-David regular measure whose $s$-Riesz transform is bounded in $\mathbb{R}^d$.  The deduction of the uniform rectifiability property from the classification of reflectionless measures was carried out in detail for the Cauchy transform in our recent lecture notes \cite{JN1}.  The higher dimensional case can be derived similarly by employing the deep geometric constructions of David and Semmes in their monograph \cite{DS}.  The David-Semmes problem was solved in the case when $s=1$ and $d=2$ by Mattila, Melnikov, and Verdera \cite{MMV}.  The case when $s=d-1$ was recently settled by Nazarov, Tolsa and Volberg in \cite{NTV12}.

In \cite{JN1}, we showed that all Ahlfors-David regular reflectionless measures for the Cauchy transform coincide with a constant multiple of the $1$-dimensional Hausdorff measure of a line.

If $s\in (d-1,d)$ we can answer Question \ref{reflconj} in the affirmative.  The following proposition, combined with Theorem \ref{reducerefl}, yields a new proof of the Eiderman-Nazarov-Volberg theorem.

\begin{prop}\label{onlytrivmeas}  There are no non-trivial reflectionless measures for the $s$-dimensional Riesz transform if $s\in (d-1,d)$.
\end{prop}

%The scheme of the proof in the Eiderman-Nazarov-Volberg theorem was one of the main technical innovations involved in Nazarov, Tolsa, and Volberg's result \cite{NTV12}.  Thus, the new proof of this theorem is of interest in its own right.

We are a long way from understanding the reflectionless measures for the Riesz transform well enough to answer Question \ref{reflconj} in either the affirmative or negative, but at least we are able to show that their structure is very different from reflectionless measures for general CZOs.

\begin{thm}\label{nodense}  Let $s\in (0,d)$.  The support of a reflectionless measure for the $s$-Riesz transform is nowhere dense.
\end{thm}

%The $s$-dimensional Riesz transform of a measure $\mu$ is the Cald\'{e}ron-Zygmund operator with kernel $K(x) = \tfrac{x}{|x|^{s+1}}$......
%\begin{thm}\label{mainthm}\cite{ENV}  Let $s\in (d-1,d)$.  Suppose that $\mu$ is an $s$-dimensional measure with bounded $s$-dimensional Riesz transform.  Then $\mu\equiv 0$.
%\end{thm}

%Prat \cite{Pra} had previously shown that this theorem also holds in the range $s\in (0,1)$ by adapting the Menger curvature method.  It is a well known open problem to extend this theorem to $s\in (1,d-1)$, $s\not\in \mathbb{Z}$.  In this paper, we deduce Theorem \ref{mainthm} from the following proposition.

%The main technical ingredient in this paper, the Collapse Lemma (Section \ref{colsec}), is valid for all $s\in (0,d)$, integer or not.  This lemma states, broadly speaking, that a reflectionless measure is porous near points where its associated Riesz transform is large.  Furthermore, this lemma can be easily extended to a broad class of singular integral operators.  The restriction in Proposition \ref{onlytrivmeas} to $s\in (d-1,d)$ is due to a reliance on a certain maximum principle for the fractional Laplacian operator.

%As was shown in \cite{JN1}, reflectionless measures can also be used to study problems of integer $s$.................

\section{Notation}  Fix an integer $d\geq  2$.  Unless stated otherwise, $s$ is a real number with $s\in (0,d)$.

For $d'\geq 1$, consider a function  $\Omega:\mathbb{R}^d \rightarrow \mathbb{C}^{d'}$, satisfying the following properties

(i)   $\lambda\Omega(x) = \Omega(\lambda x)$ for $\lambda\in \mathbb{R}$.

(ii) $|\Omega(x)|\leq 1$ for any $x\in  \mathbb{S}^{d-1}$,

(iii) there exists $\alpha\in (0,1]$ such that%such that if $x,x',y\in \mathbb{R}^d$ satisfy $|x-y|\geq 2|x-x'|$, then
$$|\Omega(x)-\Omega(x')|\leq |x-x'|^{\alpha}, \text{ for any }x,x'\in \mathbb{S}^{d-1},
$$

(iv) $\Omega(x) = -\Omega(-x)$ for any $x\in  \mathbb{S}^{d-1}$.

An $s$-dimensional CZ-kernel takes the the form $$K(x) = \frac{\Omega(\tfrac{x}{|x|})}{|x|^s}=\frac{\Omega(x)}{|x|^{s+1}}, \text{ for }x\in \mathbb{R}^d \text{ with }x\neq 0,$$ where $\Omega$ satisfies properties (i)--(iv) above.

%The final condition (homogeneity) is used to preserve the kernel when we perform blow up arguments.  Many of the results can be reformulated without this assumption by working in the closure of a given kernel function $\Omega$, that is, the set of kernels that can be obtained by taking pointwise limits $\lim_{j\rightarrow \infty}\tfrac{1}{r_j}\Omega(r_j\cdot)$ over a sequence $r_j$ of positive real numbers.  However, since the kernels that we are most interested in are homogeneous, we have opted against discussing this further.

Throughout the paper, the symbol $\Lambda $ is reserved as a parameter which governs certain regularity properties of a measure, and $\alpha>0$ is reserved to govern the smoothness of the kernel of a CZO (property (iii) above).    In what follows $C$, $c$, or $C_j$, $c_j$ (for $j\in \mathbf{N}$) are respectively large and small positive constants that may depend on $s$, $\Lambda$, $\alpha$, and $d$.  We enumerate them so that the constant with index $j$ can be chosen in terms of constants with lower indices  (for example $C_{96}$ can depend on $c_{95}$ and $C_{4}$).  Within a specific argument, if a constant $C$ or $c$ does not have an index, then it may depend on all numbered constants chosen up to that moment, and can change from line to line.  At the very least, every large constant is greater than $1$, and every small constant is less than $1$.

By a measure, we shall always mean a non-negative locally finite Borel measure.  For a measure $\mu$, $\supp(\mu)$ denotes its closed support.  The $d$-dimensional Lebesgue measure is denoted by $m_d$.

A function $f$ (either scalar or vector valued) is called Lipschitz continuous if
$$\|f\|_{\Lip} = \sup_{x,y\in \mathbb{R}^d,\; x\neq y} \frac{|f(x)-f(y)|}{|x-y|}<\infty.
$$

For an open set $U\subset \mathbb{R}^d$, $\Lip_0(U)$ denotes the set of Lipschitz continuous functions that are compactly supported in $U$.

For two $\mathbb{C}^{d'}$-valued functions $f,g\in L^2(\mu)$, we define
$$\langle f,g\rangle_{\mu} = \int_{\mathbb{R}^d} \bigl(f \cdot g\bigl) d\mu.
$$
Since a CZO maps scalar functions into $\mathbb{C}^{d'}$ valued functions, there shall be many formulae involving a mixture of scalar and vector functions.  We shall leave it to the reader to discern when a product is taken between vector or scalar valued functions.

We denote by $\mathcal{D}$ the standard lattice of half open dyadic cubes in $\mathbb{R}^d$.

The $s$-dimensional Hausdorff measure of a set $E$ is defined by
$$\mathcal{H}^s(E) = \sup_{\delta>0}\inf\bigl\{\sum_j r_j^s : E\subset \bigcup_j B(x_j, r_j),\, r_j\leq \delta\bigl\}.
$$

\section{A primer on $s$-dimensional Calder\'{o}n-Zygmund operators} \label{primer}

The results of this paper are obtained by studying the properties of $T$ from two standpoints:  as properties of a locally integrable function with respect to $m_d$, and properties of an operator acting on $L^2(\mu)$.  %In this section, we define the Riesz transform in both ways.

\begin{lem}\label{locl1} There is a constant $C_1>0$ such that for any measure $\nu$,
$$\int_{B(x,r)}\int_{B(y,R)}\frac{1}{|z-y|^s} d\nu(z)dm_d(y) \leq C_1 r^{d-s}\nu(B(x,r+R)),
$$
for any $x\in \mathbb{R}^d$, $r\in (0,\infty)$, and $R\in (0,\infty)$.
\end{lem}
In particular,  this trivial lemma implies that if $\nu$ is a finite measure, then $T(\nu)(\,\cdot\,)=\int_{\mathbb{R}^d}K(\cdot - y)d\nu(y)\in L^1_{\loc}(m_d)$.

To define $T$ as an operator in $L^2$ requires a little more effort.  For $\delta>0$, define the regularized  CZ-kernel
$$K_{\delta}(x) = \frac{\Omega(x )}{\max(\delta, |x|)^{s+1}}.
$$
For a measure $\nu$ satisfying $\int_{\mathbb{R}^d} \frac{1}{(1+|x|)^{s}}d\nu(x)<\infty$, we write $T_{\delta}(\nu)(x) = \int_{\mathbb{R}^d} K_{\delta}(x-y) d\nu(y).$  To continue our discussion we shall need to introduce a natural growth condition on a measure: $\mu$ is called $\Lambda $-nice if $\mu(B(x,r))\leq \Lambda  r^s$ for any ball $B(x,r)\subset\mathbb{R}^d$.

If $\mu$ is a $\Lambda $-nice measure, then for any $f\in L^2(\mu)$, and $x\in \mathbb{R}^d$, $$\int_{\mathbb{R}^d}\frac{|f(y)|}{(\delta+|x-y|)^s}d\mu(y)\leq \frac{C_2||f||_{L^2(\mu)}}{\delta^{s/2}}.$$
To see this, first note that by the Cauchy-Schwarz inequality, it suffices to show that $\int_{\mathbb{R}^d} \tfrac{1}{(\delta+|x-y|)^{2s}}d\mu(y)\leq \tfrac{C}{\delta^s}.$
The integral in question is smaller than $\int_{B(x,\delta)}\frac{1}{\delta^{2s}}d\mu(y) + \int_{\mathbb{R}^d\backslash B(z,\delta)}\frac{1}{|x-y|^{2s}}d\mu(y).
$
%But now $\int_{\mathbb{R}^d}|K_{\delta}(x-y)|^2d\mu(y) \leq \int_{B(x,\delta)}\frac{|x-y|^2}{\delta^{2s+2}}d\mu(y) + \int_{\mathbb{R}^d\backslash B(z,\delta)}\frac{1}{|x-y|^{2s}}d\mu(y).
%$
The first term here is trivially bounded by $\tfrac{\mu(B(x,\delta))}{\delta^{2s}}\leq \tfrac{\Lambda }{\delta^s}$, while the second term is bounded by $\tfrac{C}{\delta^s}$ due to the following standard estimate:

\begin{lem}\label{tailest}  Suppose that $\mu$ is a $\Lambda $-nice measure.  Then, for every ball $B(x,r)\subset\mathbb{R}^d$, and $\eps>0$,
$$\int\limits_{\mathbb{R}^d\backslash B(x,r)} \frac{1}{|y-x|^{s+\eps}}d\mu(\xi)\leq \frac{\Lambda (s+\eps)}{\eps} r^{-\eps}.
$$
\end{lem}

%With this lemma in hand, we return to our claim that $R_{\mu, \delta}(f)$ is bounded.  First apply the Cauchy-Schwarz inequality to estimate $$|R_{\mu, \delta}(f)(x)|\leq \Bigl(\int_{\mathbb{R}^d}|K_{\delta}(x-y)|^2d\mu(y)\Bigl)^{1/2}\|f\|_{L^2(\mu)}.$$
%But now $\int_{\mathbb{R}^d}|K_{\delta}(x-y)|^2d\mu(y) \leq \int_{B(x,\delta)}\frac{|x-y|^2}{\delta^{2s+2}}d\mu(y) + \int_{\mathbb{R}^d\backslash B(z,\delta)}\frac{1}{|x-y|^{2s}}d\mu(y).
%$
%The first term on the right hand side of this inequality is at most $\tfrac{\mu(B(x,\delta))}{\delta^{2s}}\leq \tfrac{\Lambda }{\delta^s}$, and the second term is no greater than $\tfrac{C\Lambda }{\delta^s}$ by Lemma \ref{tailest}.  We therefore see that $|R_{\mu, \delta}(f)(x)|\leq \bigl(\tfrac{C\Lambda }{\delta^s}\bigl)^{1/2}\|f\|_{L^2(\mu)}$.

As a result of this discussion, we have that the operator $T_{\mu, \delta}(f) = \int_{\mathbb{R}^d}K_{\delta}(\cdot, y)f(y)d\mu(y)$ acts boundedly $L^2(\mu)\rightarrow L^2_{\loc}(\mu)$ as soon as $\mu$ is a nice measure.   Thus, for a nice measure $\mu$, it makes sense to ask if $T_{\mu, \delta}:L^2(\mu)\rightarrow L^2(\mu)$ independently of $\delta$.

\begin{defn}\label{gooddef}A measure $\mu$ is called $\Lambda $-good (for the CZO $T$) if it is $\Lambda $-nice, and $\|T_{\mu, \delta}\|_{L^2(\mu)\rightarrow L^2(\mu)}\leq \Lambda $ for every $\delta>0$.  We say that $T_{\mu}$ is bounded in $L^2(\mu)$ if $\mu$ is $\Lambda$-good, for some $\Lambda>0$.\end{defn}

%Of course, a measure being $\Lambda $-good depends on the kernel function as well -- we should really write $(\Lambda, T)$-good -- but this dependence is suppressed.  %The Eiderman-Nazarov-Volberg theorem can be precisely stated as follows:

%\begin{thm}Let $s\in (d-1,d)$.  The only $s$-dimensional good measure for the Riesz transform is trivial (the zero measure).
%\end{thm}

If $\mu$ is a $\Lambda $-good measure, then we can define a bounded linear operator $T_{\mu}:L^2(\mu)\rightarrow L^2(\mu)$, with operator norm at most $\Lambda $, as a weak limit of the operators $T_{\mu, \delta}$, i.e. for each  $f,g\in L^2(\mu)$, $$\langle T_{\mu}(f),g\rangle_{\mu} = \lim_{\delta\rightarrow 0^+}\langle T_{\mu}(f),g\rangle _{\mu}. $$

This limiting procedure was first carried out by Mattila and Verdera \cite{MV}, and was done in detail in the papers \cite{NTV12} and \cite{JN1} for the Riesz and Cauchy transforms respectively. Thus, a discussion of the limiting procedure is deferred to an appendix.   The principal facts about the operator $T_{\mu}$ that we shall use (and that are proved in Appendix \ref{T1append}) are summarized as follows:
\begin{itemize}
\item  If, for some $f,g\in L^2(\mu)$, $\int_{\mathbb{R}^d}\int_{\mathbb{R}^d}|K(x-y)||f(x)||g(y)|d\mu(x)d\mu(y)<\infty$, then $$\langle T_{\mu}(f),g\rangle_{\mu} = \int_{\mathbb{R}^d}\int_{\mathbb{R}^d}K(x-y)f(x)g(y)d\mu(x)d\mu(y).$$
\item  If $\mu_k$ are $\Lambda $-good (resp. $\Lambda$-nice) measures which converge weakly to a measure $\mu$, then $\mu$ is $\Lambda $-good (resp. $\Lambda$-nice).
\item  If $\mu_k$ are $\Lambda $-good measures which converge to $\mu$ weakly, then $$\lim_{k \rightarrow \infty}\langle T_{\mu_k}(f),g\rangle_{\mu_k} =\langle T_{\mu}(f),g\rangle_{\mu} \text{ for any }f,g\in \Lip_0(\mathbb{R}^d).$$
\item For a $\Lambda$-good measure $\mu$, set $T_{\mu}^{\delta} = T_{\mu}-T_{\mu,\delta}$.   Then trivially, $T_{\mu}^{\delta}$ is bounded in $L^2(\mu)$ with operator norm at most $2\Lambda$.  Also, $\int_{B(x,r)}T_{\mu}^{\delta}(\chi_{B(x,r)})d\mu=0$ for any ball $B(x,r)\subset \mathbb{R}^d$.
\end{itemize}

\section{Reflectionless measures}

%Suppose that $\mu$ is a non-trivial good measure.   We make three definitions.
Before getting down to business, let us record a standard tail estimate that shall be used rather liberally in what follows.

\begin{lem}\label{taildiff}  Suppose that $\nu$ is a $\Lambda $-nice measure.  Fix $R, R'\in (0,\infty)$ with $R'\geq 2R$.  If $\dist(B(x,R), \supp(\nu)) \geq \tfrac{R'}{2}$, and $z,z'\in B(x,R)$, then
$$\int_{\mathbb{R}^d}|K(z-y)-K(z'-y)|d\nu(y) \leq C_3 \Bigl(\frac{R}{R'}\Bigl)^{\alpha}.
$$
\end{lem}

\begin{proof}
The H\"{o}lder continuity of $\Omega$ guarantees that if  $z,z'\in B(x,R)$, then
$|K(z-y)-K(z'-y)|\leq \frac{C|z-z'|^{\alpha}}{|x-y|^{s+\alpha}}%\leq \frac{CR}{|x-y|^{s+1}},
$
for any $y\in \mathbb{R}^d\backslash B(x,2R)$ (which contains $\supp(\nu))$.  But
$$\int_{\mathbb{R}^d}\frac{1}{|x-y|^{s+\alpha}}d\nu(y) \leq C\int_{\tfrac{R'}{2}}^{\infty}\frac{\nu(B(x,r))}{r^{s+\alpha}}\frac{dr}{r}\leq C\int_{\tfrac{R}{2}'}^{\infty}\frac{dr}{r^{1+\alpha}}\leq \frac{C}{R'^{\alpha}},
$$
and the lemma follows.\end{proof}

\subsection{The distribution $T_{\mu}(1)$.}\label{T1dist}  We now define $T_{\mu}(1)$ as a distribution, for a $\Lambda $-good measure $\mu$.%Suppose that $\psi\in \Lip_0(\mathbb{R}^d)$ satisfies $\int_{\mathbb{R}^d} \psi d\mu=0$.  We shall first define a distribution $R_{\mu}(1)$ acting on compactly supported Lipschitz functions with $\mu$-mean zero.

Suppose that $\psi\in \Lip_0(\mathbb{R}^d)$ with $\int_{\mathbb{R}^d} \psi d\mu =0$. Fix a ball $B$ containing $\supp(\psi)$ and fix $z\in B$.  Let $\varphi\in L^2(\mu)$ satisfy $\varphi\equiv 1$ on $2B$ and $0\leq \varphi\leq 1$ on $\mathbb{R}^d$.  Define
\begin{equation}\begin{split}\label{R1dist}\langle T_{\mu}(1), \psi\rangle_{\mu} &=\langle T_{\mu}(\varphi), \psi\rangle_{\mu}\\
& +\Big\langle\int_{\mathbb{R}^d}(1-\varphi(y))\bigl[K(\cdot-y)-K(z-y) \bigl]d\mu(y) ,\psi \Bigl\rangle_{\mu}.
\end{split}\end{equation}
%\frac{x-y}{|x-y|^{s+1}} - \frac{z-y}{|z-y|^{s+1}}
Lemma \ref{taildiff} -- applied with $\nu=[1-\varphi]\mu$, $R$ equal to the radius of $B$, and $R'=2R$ -- yields that $\int_{\mathbb{R}^d}(1-\varphi(y))\bigl|K(x-y)-K(z-y) \bigl|d\mu(y)\leq C$ for any $x\in \supp(\psi)$, and so both terms defining (\ref{R1dist}) are finite.

We now claim that $\langle T_{\mu}(1), \psi\rangle_{\mu}$ is independent of the particular choice of $B$, $\varphi$ and $z$.   %To see this, set $\psi\in \Lip_0(\mathbb{R}^d)$ with $\int_{\mathbb{R}^d}\psi d\mu=0$.
To see this, set $B_1$ and $B_2$ to be two balls with $B_1\subset B_2$ and $B_1\supset \supp(\psi)$.  Fix $z_1 \in B_1$ and $z_2\in B_2$, and consider admissible functions $\varphi_1$ and $\varphi_2$ (that is, $\varphi_j\in L^2(\mu)$, $\varphi_j\equiv 1$ on $2B_j$, and $\varphi_j\in [0,1]$ on $\mathbb{R}^d$ for each $j=1,2$).  Then
$$\langle T_{\mu}(\varphi_1), \psi\rangle_{\mu}- \langle T_{\mu}(\varphi_2), \psi\rangle_{\mu} = \langle T_{\mu}(\varphi_1-\varphi_2), \psi\rangle_{\mu}.
$$
But $\int_{\mathbb{R}^d}\int_{\mathbb{R}^d}|K(x-y)||\varphi_1(y)-\varphi_2(y)||\psi(x)|d\mu(y)d\mu(x)<\infty$, so $$\langle T_{\mu}(\varphi_1-\varphi_2), \psi\rangle_{\mu}= \int_{\mathbb{R}^d}\int_{\mathbb{R}^d}K(x-y)(\varphi_1(y)-\varphi_2(y))\psi(x)d\mu(y)d\mu(x).$$
Now, the mean zero property of $\psi$ allows us to re-write this as
$$\int_{\mathbb{R}^d}\int_{\mathbb{R}^d}[K(x-y)-K(z_1-y)](\varphi_1(y)-\varphi_2(y))\psi(x)d\mu(y)d\mu(x).$$
Courtesy of Fubini's theorem (which is applicable due to Lemma \ref{taildiff}) this integral is in tern equal to
\begin{equation}\begin{split}\nonumber\Big\langle\int_{\mathbb{R}^d}(1-\varphi_2(y))&\bigl[K(\cdot-y)-K(z_1-y) \bigl]d\mu(y) ,\psi \Bigl\rangle_{\mu}\\&- \Big\langle\int_{\mathbb{R}^d}(1-\varphi_1(y))\bigl[K(\cdot-y)-K(z_1-y) \bigl]d\mu(y) ,\psi \Bigl\rangle_{\mu}.
\end{split}\end{equation}
But $\int_{\mathbb{R}^d}(1-\varphi_2(y))\bigl[K(z_2-y)-K(z_1-y) \bigl]d\mu(y)$ is a finite constant (vector), so we  may replace $z_1$ by $z_2$ in the first term of the previous display. The claimed independence is proved, and we conclude that $T_{\mu}(1)$ is a well-defined distribution acting on compactly supported Lipschitz continuous functions with $\mu$-mean zero.

\begin{defn}\label{refldef}  A $\Lambda $-good measure $\mu$ is said to be \textit{reflectionless} (for a CZO $T$) if $$\langle T_{\mu}(1), \psi \rangle_{\mu}=0, $$
for all $\psi\in \Lip_0(\mathbb{R}^d)$ with $\int_{\mathbb{R}^d} \psi d\mu=0$.
\end{defn}

%We call a measure reflectionless (for a CZO $T$) if it is $\Lambda$-good reflectionless (for $T$) for some $\Lambda>0$.

\subsection{A weak convergence result for the distribution $T_{\mu}(1)$}

We next prove a weak convergence result for the distribution $T_{\mu}(1)$ in a sufficient amount of generality to be used at several points in the sequel.

Consider two sets of functions $$\Phi^{\mu}_R =\Bigl\{\psi\in \Lip_0(B(0, R)): \|\psi\|_{\Lip} < 1\;\text{ and } \int_{B(0,R)}\psi d\mu=0\Bigl\},$$
and $$\Phi^{\mu} =\Bigl\{\psi\in \Lip_0(\mathbb{R}^d): \|\psi\|_{\Lip} < 1\;\text{ and } \int_{\mathbb{R}^d}\psi d\mu=0\Bigl\} = \bigcup_{R>0}\Phi_R^{\mu}.$$

\begin{lem} \label{1closetophi} Let $R>0$, and $R'\geq2R$.  Suppose that $\mu$ is a $\Lambda $-good measure.  If $\varphi\in L^2(\mu)$ satisfies $\varphi\equiv 1$ on $B(0, R')$, and $0\leq \varphi\leq 1$ in $\mathbb{R}^d$,  then
$$|\langle T_{\mu}(1), \psi\rangle_{\mu}-\langle T_{\mu}(\varphi), \psi\rangle_{\mu}|\leq C_4R^{s+1}\Bigl(\frac{R}{R'}\Bigl)^{\alpha},
$$
for all $\psi\in \Phi^{\mu}_R$.
\end{lem}

\begin{proof}  With $B=B(0,R)$ and $z\in B$, the function $\varphi$ is admissible for the definition of the distribution $T_{\mu}(1)$.  Therefore, $\langle T_{\mu}(1), \psi\rangle_{\mu}-\langle T_{\mu}(\varphi), \psi\rangle_{\mu}$ is equal to
$$\int_{B(0,R)}\psi(x)\int_{\mathbb{R}^d}(1-\varphi(y))[K(x-y)-K(z-y)]d\mu(y)d\mu(x).
$$
Thus, Lemma \ref{taildiff} yields that this quantity is at most $C\|\psi\|_{L^1(\mu)}\bigl(\tfrac{R}{R'}\bigl)^{\alpha}$ in absolute value.  But, with $\omega\in \mathbb{R}^d$ satisfying $|\omega|=R$, we have
$$\int_{B(0,R)}\!\!\!|\psi|d\mu = \int_{B(0,R)}\bigl|\psi(z)-\psi(\omega)\bigl|d\mu(z)\leq 2R\mu(B(0,R))\leq CR^{s+1},
$$
as required.%which is bounded by $\mu(B(0,R))\text{osc}_{B(0,R)}\psi\leq \Lambda  R^s\cdot 2R$.
\end{proof}

\begin{lem}\label{smalldiff}  Suppose that $\mu_k$ are $\Lambda $-good measures that converge weakly to a measure $\mu$ (and so $\mu$ is $\Lambda $-good as well).  Let $\gamma_k$ and $\widetilde{R}_k$ be sequences of non-negative numbers satisfying $\gamma_k\rightarrow 0 $, and $\widetilde{R}_k\rightarrow \widetilde{R}\in (0,\infty]$, as $k\rightarrow\infty$.

If $|\langle T_{\mu_k}(1), \psi\rangle_{\mu_k}|\leq \gamma_k$ for all $\psi\in \Phi_{\widetilde{R}_k}^{\mu_k}$.  Then $\langle T_{\mu}(1), \psi\rangle_{\mu}=0$ for all $\psi\in \Phi_{\widetilde{R}}^{\mu}$. (Here $\Phi_{\widetilde{R}}^{\mu}=\Phi^{\mu}$ if $\widetilde{R} = \infty$.)
\end{lem}

\begin{proof} If $\mu(B(0, \widetilde{R}))=0$, then there is nothing to prove, so assume that $\mu(B(0, \widetilde{R}))>0$.  Fix $\psi\in \Phi_{\widetilde{R}}^{\mu}$.  Then there exists $R\in (0, \widetilde{R})$ with $\supp(\psi)\subset B(0, R)$, and $B(0,R)\subset B(0, \widetilde{R}_k)$ for sufficiently large $k$.  Suppose now that $\rho\in \Lip_0(B(0,R))$, $\rho\geq 0$, and $\|\rho\|_{L^1(\mu)}=1$.  If $k$ is large enough, then $\|\rho\|_{L^1(\mu_k)}\geq\tfrac{1}{2}$.  For these $k$, define $\psi_k = \psi - \lambda_k \rho$, where $\lambda_k = \bigl(\int_{\mathbb{R}^d}\rho\, d\mu_k\bigl)^{-1}\int_{\mathbb{R}^d} \psi \,d\mu_k$.  Note that $\lambda_k\rightarrow 0$ as $k\rightarrow \infty$.  Thus $\psi_k\in \Phi_{R}^{\mu_k}\subset \Phi_{\widetilde{R}_k}^{\mu_k}$ for sufficiently large $k$, and so $|\langle T_{\mu_k}(1), \psi_k\rangle_{\mu_k}|\leq \gamma_k$.

Let $\eps>0$.  Then by Lemma \ref{1closetophi}, we may pick $\varphi\in \Lip_0(\mathbb{R}^d)$ such that
$|\langle T_{\mu_k}(1), \psi_k\rangle_{\mu_k}-\langle T_{\mu_k}(\varphi), \psi_k\rangle_{\mu_k}|\leq \eps$ for all $k$ sufficiently large to guarantee that $\psi_k\in \Phi_{R}^{\mu_k}$, and also $|\langle T_{\mu}(1), \psi\rangle_{\mu}-\langle T_{\mu}(\varphi), \psi\rangle_{\mu}|\leq\eps$.

On the other hand, $\lim_{k\rightarrow\infty}\langle T_{\mu_k}(\varphi), \psi\rangle_{\mu_k} = \langle T_{\mu}(\varphi), \psi\rangle_{\mu}$.  %In addition, $|\langle T_{\mu_k}(\varphi), \rho\rangle_{\mu_k}|\leq \Lambda \|\varphi\|_{L^2(\mu_k)}\|\rho\|_{L^2(\mu_k)}$, and the right hand side of this inequality is bounded independently of $k$ (indeed, it converges to $\Lambda \|\varphi\|_{L^2(\mu)}\|\rho\|_{L^2(\mu)}$ as $k\rightarrow \infty$).
Consequently, $\lim_{k\rightarrow\infty}|\langle T_{\mu_k}(\varphi), \lambda_k\rho\rangle_{\mu_k}|=\bigl[\lim_{k\rightarrow\infty}\lambda_k\bigl]\cdot |\langle T_{\mu}(\varphi), \psi\rangle_{\mu}|=0$.

Bringing everything together, we see that $|\langle T_{\mu}(1), \psi\rangle_{\mu}|\leq 3\eps$, from which the lemma follows.\end{proof}

An immediate corollary of this lemma is the following useful fact.

\begin{cor}\label{reflweaklim}  Suppose that $\mu_k$ are $\Lambda $-good reflectionless measures that converge weakly to a measure $\mu$.  Then $\mu$ is a $\Lambda $-good reflectionless measure.
\end{cor}

\begin{proof} We have already seen that $\mu$ is $\Lambda $-good.  Choose $\gamma_k=0$, and $\widetilde{R}_k=k$.  Then the assumptions of the lemma are satisfied, so $\langle T_{\mu}(1), \psi\rangle_{\mu}=0$ for all $\psi \in \Phi^{\mu}$.
\end{proof}

%The definition is well defined, and is independent of the particular choice of $B$ or $z$.

\section{Wavelet coefficients and reflectionless measures}

  Suppose that $\mu$ is a $\Lambda $-nice measure.  A sequence of functions $\psi_Q$ ($Q\in\mathcal{D}$) is called a $C$-Riesz system if there is a constant $C>0$ such that
$$\Bigl\|\sum_{Q\in \mathcal{D}} a_Q\psi_Q\Bigl\|^2_{L^2(\mu)} \leq C\sum_{Q\in \mathcal{D}}|a_Q|^2 \text{ for any }(a_Q)_{Q\in \mathcal{D}}\in \ell^2(\mathcal{D}).
$$
By a trivial duality argument, $\psi_Q$ ($Q\in \mathcal{D}$) is a $C$-Riesz system if and only if
$$\sum_{Q\in \mathcal{D}}\bigl|\langle f, \psi_Q \rangle_{\mu}\bigl|^2\leq C\|f\|_{L^2(\mu)}^2 \text{ for any }f\in L^2(\mu).
$$

Now, with each $Q\in \mathcal{D}$ associate a set of functions $\Psi_Q$.  Then we say that $\Psi_Q$ ($Q\in \mathcal{D}$) is a $C$-Riesz family if, for any choice of $\psi_Q\in \Psi_Q$, the system $\psi_Q$ ($Q\in \mathcal{D}$) is a $C$-Riesz system.

We shall be interested in a particular family.  Let $A>10\sqrt{d}$.  For each $Q\in \mathcal{D}$ define
%\begin{equation}\begin{split}\nonumber\Psi^{\mu}_{Q,A} = \Bigl\{\psi\in \Lip_0\bigl(B\bigl(z_Q,& \tfrac{A\ell(Q)}{2}\bigl)\bigl):\int_{\mathbb{R}^d}\psi d\mu=0, \text{ and }\\
%&\|\psi\|_{\Lip}< \frac{1}{\ell(Q)\sqrt{\mu(B_Q, A\ell(Q))}}\Bigl\}.\end{split}\end{equation}
\begin{equation}\begin{split}\nonumber\Psi^{\mu}_{Q,A} = \Bigl\{\psi\in \Lip_0\bigl(B\bigl(x_Q,& A\ell(Q)\bigl)\bigl):\int_{\mathbb{R}^d}\psi d\mu=0, \,%\text{ and }\\
\|\psi\|_{\Lip}< \frac{1}{\ell(Q)^{1+\tfrac{s}{2}}}\Bigl\}.\end{split}\end{equation}

%If, for a cube $Q\in \mathcal{D}$, $\mu(B(x_Q, A\ell(Q)))=0$, then we instead set $\Psi^{\mu}_{Q,A} =\{0\}$.

\begin{lem}\label{systemlem} $\Psi_Q^{\mu}$ $(Q\in \mathcal{D})$ is a $C_5A^{d+2+\tfrac{3s}{2}}$-Riesz family.
\end{lem}

The proof is a rather standard piece of harmonic analysis, and as such, is relegated to an appendix.  We now bring reflectionless measures into the picture.

\subsection{Riesz systems and reflectionless measures}

Suppose now that $\mu$ is a $\Lambda $-good measure.  For $A'>2A$, define $$\Theta^{A,A'}_{\mu}(Q) = \inf_{E\supset B(x_Q, A'\ell(Q'))}\sup_{\psi\in \Psi_A^{\mu}(Q)}\Bigl[\frac{1}{\ell(Q)^{\tfrac{s}{2}}}\bigl|\langle T_{\mu}(\chi_E), \psi\rangle_{\mu}\bigl|\Bigl].  $$

We have the following dichotomy,

\begin{prop}\label{lowbdtheta}  Either there exists a non-trivial $\Lambda$-good reflectionless measure, or the following property holds
\begin{equation}\tag{$*$}\begin{split}
   &\textit{For every }\Delta>0 \textit{ there exists }A=A(\Delta)>0, A'=A'(\Delta)>2A,\\
   & \textit{and }\eps=\eps(\Delta)>0, \textit{ such that for every }\Lambda\textit{-good measure }\mu,\\
   &\textit{if } Q\in \mathcal{D}\textit{ satisfies }\mu(Q)\geq \Delta \ell(Q)^s,\textit{ then } \Theta_{A,A'}^{\mu}(Q)\geq \eps.
\end{split}\end{equation}
\end{prop}

%For each $\Delta>0$, there exists $A=A(\Delta)>0$ and $A'=A'(\Delta)>A$ and $\eps=\eps(\Delta)>0$ such that for any $Q\in \mathcal{D}$ with $\mu(Q)\geq \Delta \ell(Q)^s$, $\Theta_{A,A'}^{\mu}(Q)\geq \eps$.

%We conjecture that property ($*$) holds for any $s\in (0,d)$, $s\not\in \mathbb{Z}$, but we cannot prove this for the following reason.

%Our main aim is now to prove the following proposition.

%\begin{prop}  Property $\operatorname{(}*\operatorname{)}$ fails to hold if and only if there exists a non-trivial $\Lambda $-good reflectionless measure.
%\end{prop}

%In the next section we lay the groundwork needed to prove Proposition \ref{lowbdtheta}.

%We remark that the $d$-dimensional Lebesgue measure of a ball is a good measure for any $s$-dimensional CZO if $s\in (0,d)$.  Therefore, the statement that $\operatorname{(}*\operatorname{)}$ is not vacuous.

\begin{proof}
%First observe that there are many non-trivial $\Lambda $-good measures.
Suppose first that ($*$) fails to hold.  Temporarily fix $A\in \mathbb{N}$, $A>10\sqrt{d}$.  Then for some small $\Delta>0$ (independent of $A$), and each $k\in \mathbb{N}$, $k>2A$, there is a $\Lambda $-good measure $\mu_k$, a dyadic cube $Q_{k}\in \mathcal{D}$ with $\mu_k(Q_k)\geq \Delta \ell(Q_k)^s$, and a set $E_k\supset B(x_{Q_k}, k\ell(Q_k))$ such that $$\ell(Q_k)^{-\tfrac{s}{2}} \bigl|\langle T_{\mu_k}(\chi_{E_k}), \psi\rangle_{\mu_k}\bigl|\leq \frac{1}{k}, $$
for all $\psi\in \Psi_{A}^{\mu_k}(Q_k)$.
%\frac{\sqrt{\mu_k(x_{Q_k}, A\ell(Q_k)))}}{\ell(Q_k)^s}

Consider the measure given by $\tilde\mu_k(F) = \tfrac{\mu(\ell(Q_k)\cdot F+ x_{Q_k})}{\ell(Q_k)^s}$, $F\subset \mathbb{R}^d$ Borel, and set $\widetilde{E}_k = \tfrac{E_k - x_{Q_k}}{\ell(Q_k)}$.  Then  $\tilde{\mu}_k ([-\tfrac{1}{2}, \tfrac{1}{2}]^d)\geq \Delta$, and
 $$ \bigl|\langle T_{\tilde{\mu}_k}(\chi_{\widetilde{E}_k}), \psi\rangle_{\tilde{\mu}_k}\bigl|\leq \frac{1}{k} \text{ for any } \psi\in \Phi_{A}^{\tilde{\mu}_k}.$$
(Here the homogeneity in the CZ-kernel is used.)  Each $\tilde\mu_k$ is $\Lambda $-good, so by passing to a subsequence we may assume that the measures $\tilde\mu_k$ converge weakly to a $\Lambda $-good measure $\mu^{(A)}$.  Note that $\mu^{(A)}([-\tfrac{1}{2}, \tfrac{1}{2}]^d)\geq \Delta$.  Furthermore, Lemma \ref{1closetophi} (applied with $R=A$ and $R'=k$) ensures that
$$ \bigl|\langle T_{\tilde{\mu}_k}(1), \psi\rangle_{\tilde{\mu}_k}\bigl|\leq \frac{1}{k}+\frac{CA^{s+1+\alpha}}{k^{\alpha}} \text{ for any } \psi\in \Phi_{A}^{\tilde{\mu}_k}.$$
Thus, applying  Lemma \ref{smalldiff} to the convergent subsequence of $(\tilde\mu_k)_k$, with $\gamma_k = \tfrac{1}{k}+\tfrac{CA^{s+1+\alpha}}{k^{\alpha}}$, and $\tilde{R}_k = A$, yields $$ \langle T_{\mu^{(A)}}(1), \psi\rangle_{\mu^{(A)}}=0 \text{ for any } \psi\in \Phi_{A}^{\mu^{(A)}}.$$

Now let $A\rightarrow \infty$.  Since each measure $\mu^{(A)}$ is $\Lambda $-good, there is a subsequence $A_{\ell}\rightarrow \infty$ so that $\mu^{(A_{\ell})}$ converges to a $\Lambda $-good measure $\mu$ with $\mu([-\tfrac{1}{2}, \tfrac{1}{2}]^d)\geq \Delta$.  Applying Lemma \ref{smalldiff} once again, with $\gamma_{\ell} = 0$ and $\widetilde{R}_{\ell} = A_{\ell}$, yields that $\mu$ is reflectionless.

On the other hand, suppose that $\mu$ is a non-trivial $\Lambda$-good reflectionless measure.  Choose a cube $Q\in \mathcal{D}$ for which $\mu(Q)>0$.  By an appropriate translation and resealing, we may suppose that $Q=[0,1)^d$.  Fix $A>10\sqrt{d}$, and $A'>2A$.  Suppose that $E= B(x_Q, R)$ for $R>2A'$.  Then for any $\psi\in \Phi^{\mu}_{2A}$ (a set of functions containing $\Psi_{A}^{\mu}(Q)$), Lemma \ref{1closetophi} yields
$|\langle T_{\mu}(\chi_{E}), \psi\rangle_{\mu}|\leq \frac{CA^{s+1+\alpha}}{R^{\alpha}}.$ % So
%$$|\langle T_{\mu}(\chi_{E}), \psi\rangle_{\mu}|\sqrt{\mu(B(x_Q, A))}\leq \frac{CA^{\tfrac{3s}{2}+1+\alpha}}{(MA')^{\alpha}}\text{ for any }\psi\in \Psi^{\mu}_{A}(Q).
%$$

Letting $R\rightarrow \infty$, we see that $\Theta^{A,A'}_{\mu}(Q)=0$,  which precludes the possibility that ($*$) holds.\end{proof}

\subsection{The proof of Theorem \ref{reducerefl}}

Recall that $\mu$ is called $s$-dimensional if $\mu$ is $\Lambda $-nice, and $\mathcal{H}^s(\supp(\mu))<\infty$.  In this case $\mu(\mathbb{R}^d)\leq \Lambda  \mathcal{H}^s(\supp(\mu))<\infty.$  Let us first recall Theorem \ref{reducerefl}.

\begin{thm}\label{startriv}  Suppose that there does not exist a non-trivial reflectionless measure.  If $\mu$ is an $s$-dimensional $\Lambda$-good measure, then $\mu\equiv 0$.
\end{thm}

\begin{proof} Since there is no non-trivial reflectionless measure, Proposition \ref{lowbdtheta} yields that property $(*)$ holds.  If $\mu(\mathbb{R}^d)>0$, then there exists $P\in \mathcal{D}$ with $\mu( P)>0$.  Fix $\Delta \in (0,1)$, and set $\mathcal{D}_{\Delta} = \{Q\in \mathcal{D}: \mu(Q)\geq \Delta\ell(Q)^s\}$.  For each $Q\subset P$ with $Q\in \mathcal{D}_{\Delta}$, there exists $\psi_Q\in \Psi_{A}^{\mu}(Q)$ such that %$$\ell(Q)^{-\tfrac{s}{2}}|\langle T_{\mu}(\chi_{B(x_P, 2A'\ell(P))}), \psi_Q\rangle_{\mu}|\geq \frac{\eps(\Delta)}{2}.$$ Employing this inequality in conjunction with the $\Lambda$-niceness of $\mu$ yields
$$\ell(Q)^{s}\leq C(\Delta)\bigl|\langle T_{\mu}(\chi_{B(x_P, 2A'\ell(P))}), \psi_Q\rangle_{\mu}\bigl|^2,
$$
where $C(\Delta)>0$ is a constant that may depend on $d$, $s$, $\alpha$, $\Lambda$ ,and $\Delta$ (and can change from line to line for the rest of this proof).

Recalling that $\Psi_A^{\mu}(Q)$ ($Q\in \mathcal{D}$) is a $C(\Delta)$-Riesz system (Lemma \ref{systemlem}), we infer that
\begin{equation}\label{sumcarl}\sum_{Q\subset P: Q\in \mathcal{D}_{\Delta}} \ell(Q)^s \leq C(\Delta)||T_{\mu}(\chi_{B(x_P, 2A'\ell(P))})||_{L^2(\mu)}^2\leq C(\Delta)\ell(P)^s.\end{equation}
Consequently, there exists a dyadic fraction $\ell_0>0$ such that $$\sum_{Q\subset P: Q\in \mathcal{D}_{\Delta}, \,\ell(Q)\leq \ell_0} \ell(Q)^s\leq \Delta.$$
Inasmuch as $\mathcal{H}^s(P\cap \supp(\mu))<\infty$, there are cubes $Q_j\in \mathcal{D}$, with $Q_j\subset P$, $\ell(Q_j)\leq \ell_0$, $P\,\cap \,\supp(\mu) \subset \cup_j Q_j$, and $\sum_j \ell(Q_j)^s <12^d(\mathcal{H}^s(P\cap \supp(\mu))+\Delta)$\footnote{Indeed, let $(B_j)_j$ be a cover of $P\cap \supp(\mu)$ by balls of radius $r_j\leq \tfrac{\ell_0}{4}$ with $B_j\cap (P\cap \supp(\mu))\neq \varnothing$ and $\sum_j r_j^s<\mathcal{H}^s(P\cap \supp(\mu))+\Delta$.  Each set $B_j\cap P$ is covered by at most $3^d$ cubes $Q\in \mathcal{D}$ with $Q\subset P$ and $\ell(Q)\in (2r_j,4r_j]$.  The collection of cubes obtained in this way satisfy the required properties.}. But then,
$$\sum_j \mu(Q_j) = \sum_{j: Q_j\not\in \mathcal{D}_{\Delta}}\mu(Q_j) +\sum_{j: Q_j\in \mathcal{D}_{\Delta}}\mu(Q_j)  \leq \Delta \sum_j \ell(Q_j)^s+\Delta,
$$
which is bounded by $12^d\Delta[\mathcal{H}^s(P\cap \supp(\mu))]+(2^{s}+1)\Delta$.  Letting $\Delta\rightarrow 0$, we conclude that $\mu(P)=0$.
%The right hand side of this inequality is bounded by $C(\Delta)\ell(P)^s.$
%The Chebyshev inequality now yields, $$\mu\bigl(\bigl\{x\in P: x \text{ lies in } N \text{ squares } 3Q, Q\in \mathcal{D}_{\Delta}, Q\subset P\bigl\}\bigl)\leq \frac{C(\Delta)}{N}\ell(P)^s.$$
%%We conclude that, for each $\Delta>0$, $$\mu\bigl(\bigl\{x\in \mathbb{R}^d: x \text{ lies in infinitely many squares }3Q, \, Q\in \mathcal{D}_{\Delta}\bigl\}\bigl)=0.$$
%Whence, $\lim_{r\rightarrow 0}\frac{\mu(B(x,r))}{r^s}=0$ for $\mu$-almost every $x\in \mathbb{R}^d$.
%Now fix $\eps >0$.  Appealing to Egoroff's theorem, we find set a $K\subset \supp(\mu)$ with $\mu(K)\geq \tfrac{1}{2} \mu(\mathbb{R}^d)$ such that $\frac{\mu(B(\, \cdot\,,r))}{r^s} \rightarrow 0$ uniformly as $r\rightarrow 0$ on  $K$.  Inasmuch as $\mathcal{H}^s(\supp(\mu))<\infty$, there is a cover of $K$ by balls $B_j$ of radius $r_j$, with $\sum_j r_j^s<\mathcal{H}^s(\supp(\mu))+\eps$, and $\mu(B_j)\leq \eps r_j^s$ for each $j$.  Thus
%$$\mu(E) \leq \sum_j \mu(B_j) \leq \eps \sum_j r_j^s \leq \eps\mathcal{H}^s(\supp(\mu))+ \eps^2.
%$$
%But since $\eps$ was chosen arbitrarily, we conclude that $\mu$ is trivial.
\end{proof}

\section{The function $\T1_{\mu}(1)$}  In order to study the finer properties of a reflectionless measure $\mu$, we shall require a point-wise defined function $\T1_{\mu}(1)$.  The definition may initially seem rather cumbersome, but as we shall see, it is particularly well suited to the study of reflectionless measures.

Suppose that $\mu$ is a nontrivial $\Lambda $-good measure. Fix a ball $B'$ with $\mu(B')>0$, along with a nonnegative function $\eta_{B'}\in \Lip_0(B')$ satisfying $\int_{B'}\eta_{B'}d\mu=1$.  Choose a ball $B$ containing $\supp(\psi)$ with $B\supset B'$.  Then set $\varphi\in L^2(\mu)$ with compact support, satisfying $\varphi\equiv 1$ on $2B$, $0\leq \varphi\leq 1$ on $\mathbb{R}^d$.  Define $\T1_{\mu,\delta}(1)(x)$ for $x\in B$ by
\begin{equation}\begin{split}\nonumber\T1_{\mu,\delta}(1)&(x)= T_{\mu,\delta}(\varphi)(x) - \int_{B'}\eta_{B'} T_{\mu}(\varphi)d\mu \\
&+ \int_{\mathbb{R}^d}(1-\varphi(y))\int_{B'}\eta_{B'}(z)\bigl[K_{\delta}(x-y)-K(z-y) \bigl]d\mu(z)d\mu(y).
\end{split}\end{equation}
Lemma \ref{taildiff} yields that $\int_{\mathbb{R}^d\backslash 2B}|K_{\delta}(x-y)-K(z-y)|d\mu(y)\leq C$ for any $x,z\in B$, and so $\widetilde{T}_{\mu, \delta}(1)$ is finite.  The value of $\widetilde{T}_{\mu,\delta}(1)(x)$ is independent of the particular choices of $B$ and $\varphi$.  This follows (like several calculations in this section) in the same manner as the fact that the distribution $T_{\mu}(1)$ is independent of the choices of $B$ and $\varphi$ in its definition (see Section \ref{T1dist}).  We therefore leave the verification to the reader.

The function $\T1_{\mu,\delta}(1)$ is H\"{o}lder continuous:

\begin{lem}\label{Lipest} Suppose that $\mu$ is a $\Lambda $-good measure.  There is a constant $C_6>0$ such that for any $\delta>0$, and $x,x'\in \mathbb{R}^d$, $$|\T1_{\mu, \delta}(1)(x)-\T1_{\mu,\delta}(1)(x')|\leq \frac{C_6|x-x'|^{\alpha}}{\delta^{\alpha}}\max\Bigl(1,\frac{|x-x'|}{\delta}\Bigl)^{1-\alpha}.$$\end{lem}

\begin{proof}With a careful application of Fubini's theorem, it is not difficult to see that
$$|\T1_{\mu, \delta}(1)(x)-\T1_{\mu,\delta}(1)(x')|\leq \int_{\mathbb{R}^d}|K_{\delta}(x-y)-K_{\delta}(x'-y)|d\mu(y).
$$
But, for $x,x'\in \mathbb{R}^d$,
$$|K_{\delta}(x)-K_{\delta}(x')|\leq \frac{C|x-x'|}{(\delta+\min(|x|, |x'|))^{s+1}} + \frac{C|x-x'|^{\alpha}}{(\delta+\min(|x|, |x'|))^{s+\alpha}}.
$$
%But now, $|K_{\delta}(x-y)-K_{\delta}(x'-y)|\leq C\frac{|x-x'|}{\delta^{s+1}}$ for  $y\in B(x,\delta)\cup B(x',\delta)$.  If $y$ lies outside of this set, we instead have $$|K_{\delta}(x-y)-K_{\delta}(x'-y)|\leq \frac{C|x-x'|}{\min(|x-y|, |x'-y|)^{s+1}}.
%$$
Replacing $x$ by $x-y$ and $x'$ by $x'-y$, and integrating the resulting estimate over $y\in \mathbb{R}^d$ with respect to $\mu$ yields the desired bound.\end{proof}

Consider two different choices of pairs $(B', \eta_{B'})$ and $(B'', \eta_{B''})$, with $\mu(B')>0,$ $\mu(B'')>0$, $\eta_{B'}\in\Lip_0(B')$, $\eta_{B''}\in \Lip_0(B'')$, and $\int_{B'}\eta_{B'}d\mu=\int_{B''}\eta_{B''}d\mu=1$.  If the ball $B$ contains both $B'$ and $B''$, and $\varphi\equiv 1$ on $2B$, then the difference between $\T1_{\mu, \delta}(1)(x)$ ($x\in B$) defined with $(B', \eta_{B'})$, and $\T1_{\mu,\delta}(1)(x)$ defined with $(B'', \eta_{B''})$, is equal to $\langle T_{\mu}(1), \eta_{B''}-\eta_{B'}\rangle_{\mu}$.  %\begin{equation}\begin{split}\nonumber\langle & T_{\mu}(\varphi), \eta_{B''}-\eta_{B'}\rangle_{\mu} +\int_{\mathbb{R}^d}(1-\varphi(y))\int_{\mathbb{R}^d}K(z-y)[\eta_{B''}(z)-\eta_{B'}(z)]d\mu(z)d\mu(y)\\
%&=
%\langle T_{\mu}(1), \eta_{B''}-\eta_{B'}\rangle_{\mu}.\end{split}\end{equation}
Consequently, if $\mu$ is reflectionless, then this difference is equal to zero, and the value of $\T1_{\mu,\delta}(1)$ is independent of the choice of $B'$ and $\eta_{B'}$.

For $\Lambda$-good $\mu$, we shall now pass to the limit in $\T1_{\mu, \delta}(1)$ as $\delta$ tends to zero in two senses: in $L^2_{\text{loc}}(\mu)$, and $L^1_{\text{loc}}(m_d)$.

For a test function $\psi\in L^2(\mu)$ with compact support, choose the ball $B$ containing $\supp(\psi)$ with $B\supset B'$.  Then set $\varphi\in L^2(\mu)$, with $\varphi\equiv 1$ on $2B$, $0\leq \varphi\leq 1$ on $\mathbb{R}^d$.

Note that $\lim_{\delta\rightarrow 0} \langle T_{\mu,\delta}(\varphi),\psi\rangle_{\mu} = \langle T_{\mu}(\varphi),\psi\rangle_{\mu}$, and as long as $\delta$ is smaller than the radius of $B$, $K_{\delta}(x-y)=K(x-y)$ for any $x\in B$ and $y\in \mathbb{R}^d\backslash 2B$.  Thus
$$\lim_{\delta\rightarrow 0}\langle \T1_{\mu,\delta}(1), \psi \rangle_{\mu} = \langle \T1_{\mu}(1),\psi \rangle_{\mu},
$$
where $\T1_{\mu}(1)$ is defined by
\begin{equation}\begin{split}\label{R1func}& \T1_{\mu}(1) = T_{\mu}(\varphi)-\langle T_{\mu}(\varphi), \eta_{B'}\rangle_{\mu}\\
& + \int_{\mathbb{R}^d}\!\!(1-\varphi(y))\!\!\int_{B'}\!\!\eta_{B'}(z)\bigl[K(x\!-\!y)\!-\!K(z\!-\!y) \bigl]d\mu(z)d\mu(y).
\end{split}\end{equation}
%\begin{equation}\begin{split}\label{R1func}&\langle \T1_{\mu}(1), \psi\rangle_{\mu} =\langle T_{\mu}(\varphi), \psi\rangle_{\mu}-\langle T_{\mu}(\varphi), \eta_{B'}\rangle_{\mu}\cdot \langle \psi,1\rangle_{\mu}\\
%& + \Bigl\langle\int_{\mathbb{R}^d}\!\!(1-\varphi(y))\!\!\int_{B'}\!\!\eta_{B'}(z)\bigl[K(\cdot\!-\!y)\!-\!K(z\!-\!y) \bigl]d\mu(z)d\mu(y),\psi\Bigl\rangle_{\mu}.
%\end{split}\end{equation}
%All terms on the right hand side here are well defined (courtesy of the $L^2(\mu)$ boundedness of $T_{\mu}$, and Lemma \ref{taildiff}).
As the limit of $\T1_{\mu, \delta}(1)$, the value of $ \T1_{\mu}(1)$ is independent of the particular choices of $B$ and $\varphi$.      Furthermore, for any choice of $B$,  $\|T_{\mu, \delta}(\chi_{2B})\|_{L^2(\mu)}\leq C \sqrt{\mu(2B)}$ for every $\delta>0$.  %, and $\bigl|\int_{\mathbb{R}^d}(1-\varphi(y))\int_{B'}\eta_{B'}(z)\bigl[K_{\delta}(x-y)-K(z-y) \bigl]d\mu(z)d\mu(y)\bigl|$ is uniformly bounded in terms of $\Lambda $ for any $x\in B$ (using the goodness of $\mu$, and Lemma \ref{taildiff}).
Thus $\sup_{\delta>0}\|\T1_{\mu,\delta}(1)\|_{L^2(B,\mu)}\leq C(B,B')$ for any ball $B$,  from which it follows that %coincides with an $L^2(B, \mu)$ function when paired against a test function $\psi\in \Lip_0(B)$.  By the density of $\Lip_0(B)$ functions in $L^2(B,\mu)$, we therefore have
 $\T1_{\mu}(1) \in L^2_{\text{loc}}(\mu)$.% in particular is well defined $\mu$-a.e.

For $\psi\in \Lip_0(\mathbb{R}^d)$, Fubini's theorem yields
 \begin{equation}\label{R1posdiff}\langle \T1_{\mu}(1), \psi\rangle_{\mu} = \langle T_{\mu}(1), \bigl[\psi - \langle \psi, 1\rangle_{\mu}\eta_{B'}\bigl]\rangle_{\mu}.
 \end{equation}
% But the distribution $T_{\mu}(1)$ is independent of the choice of $B$ and $\varphi$.
Since $\bigl[\psi - \langle \psi, 1\rangle_{\mu}\eta_{B'}\bigl]$ has mean zero, if $\mu$ is reflectionless then $\langle \T1_{\mu}(1), \psi\rangle_{\mu}=0$ for any $\psi\in \Lip_0(\mathbb{R}^d)$.  Consequently, for a reflectionless measure $\mu$, $\T1_{\mu}(1)=0$ for $\mu$-almost every $x\in \mathbb{R}^d$.

Returning to the case when $\mu$ is merely $\Lambda$-good, let us now pass to the limit in $\T1_{\mu, \delta}(1)$ as $\delta$ tends to zero in $L^1_{\text{loc}}(m_d)$.  To do this, first recall  that Lemma \ref{locl1} yields that $T(\varphi\mu)\in L^1_{\text{loc}}(m_d)$ (using the compact support of $\varphi$).  Thus, if we define
\begin{equation}\begin{split}\label{R1Leb}&\T1_{\mu}(1)(x) = T(\varphi \mu)(x) - \int_{B'}\eta_{B'}T_{\mu}(\varphi) d\mu \\
&+\int_{\mathbb{R}^d}(1-\varphi(y))\int_{B'}\eta_{B'}(z)\bigl[K(x-y)-K(z-y) \bigl]d\mu(z)d\mu(y), \end{split}\end{equation}
then, for any ball $D$,
$\lim_{\delta\rightarrow 0}\int_{D}|\T1_{\mu,\delta}(1) - \T1_{\mu}(1)|dm_d =0.$  This convergence ensures that $\T1_{\mu}(1)$ is independent of the choices of $B$ and $\varphi$.  Furthermore, if $\mu$ is reflectionless, then $\T1_{\mu}(1)$ is independent of the choices of $B'$ and $\eta_{B'}$.

Now fix $\delta,\tau>0$ with $\tau>\delta$.  Fubini's theorem yields that, for any $z\in \mathbb{R}^d$
$$\T1_{\mu, \delta}(1)(z) - \T1_{\mu, \tau}(1)(z) = \int_{\mathbb{R}^d}[K_{\delta}(z-y)-K_{\tau}(z-y)]d\mu(y).
$$
Consequently, if $B(x,r)\subset\mathbb{R}^d$, and $z\in B(x,r)$, then
$$\T1_{\mu, \delta}(1)(z) - \T1_{\mu, \tau}(1)(z) = \int_{B(x, r+\tau)}[K_{\delta}(z-y)-K_{\tau}(z-y)]d\mu(y).
$$
Letting $\delta\rightarrow 0$ in $L^2_{\loc}(\mu)$ yields that, for $\mu$-almost every $z\in B(x,r)$,
$$\T1_{\mu}(1)(z) - \T1_{\mu, \tau}(1)(z) = T_{\mu}(\chi_{B(x,r+\tau)})(z) - T_{\mu, \tau}(\chi_{B(x,r+\tau)})(z),
$$
which equals $T_{\mu}^{\tau}(\chi_{B(x,r+\tau)})(z).$  By instead taking the limit as $\delta\rightarrow 0$ in $L^1_{\text{loc}}(m_d)$, we see that for $m_d$-almost every $z\in B(x,r)$,
$$\T1_{\mu}(1)(z) - \T1_{\mu, \tau}(1)(z) =\int_{B(x,r+\tau)}[K(z-y)-K_{\tau}(z-y)]d\mu(y).
$$
%Set $\T1_{\mu}^{\delta}(1)(x) = \T1_{\mu}(1)(x) - \T1_{\mu, \delta}(1)(x)$, which is well-defined for $\mu$-almost every $x\in \mathbb{R}^d$.  Notice that, for an open ball $D$ and $x\in D$, $$\T1_{\mu}^{\delta}(1)(x) = T_{\mu}^{\delta}(\chi_{D})+\int_{\mathbb{R}^d\backslash D} [K_{\delta}(x-y)-K_{\delta}(x-y)]d\mu(y).$$

%Thus, if $B_1$ and $B_2$ are two balls so that $B_2$ contains the $\delta$-neighborhood of $B_1$, then for $\mu$-almost every $x\in B_1$,
%$$\T1_{\mu}^{\delta}(1)(x) = T_{\mu}^{\delta}(\chi_{B_2})(x).
%$$

Our next remark of this section is a Cotlar inequality for reflectionless measures.

\begin{lem}  There exists a constant $C_7>0$ such that for any non-trivial $\Lambda$-good reflectionless measure,
$$\sup_{\delta>0}|\T1_{\mu, \delta}(1)(x)|\leq C_7, \text{ for any }x\in \mathbb{R}^d.
$$
\end{lem}

Before we prove this lemma, let us note an immediate corollary of it.

\begin{cor}\label{reflinfbd}  If $\mu$ is a non-trivial $\Lambda$-good reflectionless measure, then $\|\T1_{\mu}(1)\|_{L^{\infty}(m_d)}\leq C_{7}$.
\end{cor}

\begin{proof}[Proof of the Cotlar inequality]  The proof follows a standard path, based upon an idea of David and Mattila, see \cite{DM, NTrV}.   Let $\delta>0$, and set $B_j = B(x,2^j\delta)$.  Suppose that $\mu(B_j)\geq 2^{s+1}\mu(B_{j-1})$ for all $j\in \mathbb{Z}_+$.  For some $j'$, $\mu(B_{j'})>0$.  But then for $j>j'$ sufficiently large,
$$\mu(B_j) \geq \mu(B_{j'})2^{(s+1)(j-j')}> \Lambda 2^{sj}\delta^s,
$$
which is a contradiction.  Thus, there is a least $j\in \mathbb{Z}_+$ with $\mu(B_{j+1}) < 2^{s+1}\mu(B_{j})$.  Set $r=2^j\delta$.  Then $\mu(B(x,r))>0$, and $\mu(B(x,2r))<2^{s+1}\mu(B(x,r))$.  First note that
$$|\T1_{\mu, \delta}(1)(x) - \T1_{\mu, r}(1)(x)|\leq \int_{B(x,r)}|K_{\delta}(x-y)-K_{r}(x-y)|d\mu(y).
$$
The right hand side is trivially bounded by $2\int_{B(x,r)}\tfrac{d\mu(y)}{(\delta+|x-y|)^s}$.  But now note that
$$\sum_{0\leq \ell\leq j}\frac{\mu(B(x, 2^{\ell}\delta))}{2^{\ell s}\delta^s}\leq \mu(B(x, 2^j\delta))\sum_{0\leq \ell\leq j}\frac{1}{2^{(s+1)(j-\ell)}2^{\ell s}\delta^s}.
$$
The sum on the right hand side has size at most $\Lambda 2^{js}\sum_{0\leq \ell\leq j}\tfrac{2^{\ell}}{2^{j(s+1)}}\leq C.$  From this we conclude that $|\T1_{\mu, \delta}(1)(x) - \T1_{\mu, r}(1)(x)|\leq C.$

It remains to estimate
$$| \T1_{\mu, r}(1)(x)|= \bigl| \T1_{\mu, r}(1)(x)-\frac{1}{\mu(B(x, r))}\int_{B(x,r)}\T1_{\mu}(1)(z)d\mu(z)\bigl|.
$$
To do this, first note that by Lemma \ref{Lipest},
$$\Bigl|\T1_{\mu,r}(1)(x) - \frac{1}{\mu(B(x, r))}\int_{B(x,r)}\T1_{\mu,r}(1)(z)d\mu(z)\Bigl|\leq C.
$$
On the other hand,
$$\int_{B(x,r)}|\T1_{\mu}(1)(z)-\T1_{\mu,r}(1)(z)|d\mu(z)=\int_{B(x,r)}|T^r_{\mu}(\chi_{B(x,2r)})|d\mu.
$$
But $T^r_{\mu}$ is bounded in $L^2(\mu)$, with operator norm at most $2\Lambda$, so
$$\frac{1}{\mu(B(x,r))}\int_{B(x,r)}|T^r_{\mu}(\chi_{B(x,2r)})|d\mu\leq 2\Lambda\frac{\sqrt{\mu(B(x,r))\mu(B(x,2r))}}{\mu(B(x,r))}\leq C,
$$
where the doubling property $\mu(B(x,2r))\leq 2^{s+1}\mu(B(x,r))$ was used in the final inequality.  Bringing these estimates together proves the lemma.
\end{proof}
%\subsection{The function $\T1_{\mu}(1)\in L^1_{\loc}(m_d)$}  We shall also need to interpret $\T1_{\mu}(1)$ as a Lebesgue measurable function.  Again fix $B'$ with $\mu(B')>0$, and choose $\eta_{B'}\in \Lip_0(B')$ with $\int_{B'}\eta_{B'} d\mu=1$.  For a ball $B$ containing $B'$, choose a compactly supported function $\varphi\in L^2(\mu)$, satisfying $\varphi\equiv 1$ on $2B$ and $0\leq \varphi\leq 1$ on $\mathbb{R}^d$.  Then for $x\in B$ set
%\begin{equation}\begin{split}\label{R1Leb}&\T1_{\mu}(1)(x) = T(\varphi \mu)(x) - \int_{B'}\eta_{B'}T_{\mu}(\varphi) d\mu \\
%&+\int_{\mathbb{R}^d}(1-\varphi(y))\int_{B'}\eta_{B'}(z)\bigl[K(x-y)-K(z-y) \bigl]d\mu(z)d\mu(y). \end{split}\end{equation}
% Since $\varphi\mu$ is a finite measure, $T(\varphi \mu)\in L^1_{\text{loc}}(m_d)$.  It should not come as a surprise to the reader to learn that this definition is independent of the particular choices of $B$ and $\varphi$.  Thus $\T1_{\mu}(1)\in L^1_{\loc}(\mathbb{R}^d)$.

%If $\mu$ is reflectionless then $\T1_{\mu}(1)$, interpreted as a function in $L^1_{\loc}(m_d)$, is also independent of the choices of $B'$ and $\eta_{B'}$.

Finally, let us record a simple weak convergence result.

\begin{lem}\label{lebconv} Suppose that $\nu_j$ are non-trivial $\Lambda$-good reflectionless measures that converge weakly to $\nu$ (and so $\nu$ isa  $\Lambda$-good reflectionless measure).  Assume that $0\not\in \supp(\nu)$ and $0\not\in\supp(\nu_j)$ for all $j$.  If $\nu$ is non-trivial,  then $\T1_{\nu_j}(1)(0)\rightarrow \T1_{\nu}(1)(0)$ as $j\rightarrow \infty$.
\end{lem}

\begin{proof}
Since $\nu$ is non-trivial, there is a ball $B'$, and non-negative $\eta_{B'}\in \Lip_0(B')$ such that $\int_{B'}\eta_{B'}d\nu=1$.  Set $\lambda_j = \bigl(\int \eta_{B'}d\nu_j\bigl)^{-1}$.  Then $\lambda_j \rightarrow 1$ as $j\rightarrow \infty$.  We shall henceforth assume that $j$ is large enough so that $\lambda_j>\tfrac{1}{2}$.  Let $N>0$, and choose $\varphi_N\in \Lip_0(B(0, 2N))$ satisfying $\varphi_N\equiv 1$ on $B(0,N)$ and $0\leq \varphi_N\leq 1$ on $B(0, 2N)$.  If $N$ is large enough to ensure that $B(0, \tfrac{N}{2})\supset B'$,  then
\begin{equation}\begin{split}\nonumber &\T1_{\nu_j}(1)(0)= T(\varphi_N \nu_j )(0) - \int_{B'}\lambda_j\eta_{B'} T_{\nu_j}(\varphi_N)d\nu_j \\
&+ \int_{\mathbb{R}^d}(1-\varphi_N(y))\int_{B'}\lambda_j\eta_{B'}(z)\bigl[K(-y)-K(z-y) \bigl]d\nu_j(z)d\nu_j(y).
\end{split}\end{equation}
Notice that, for all sufficiently large $j$, $$\int_{\mathbb{R}^d}(1-\varphi_N(y))\int_{B'}\lambda_j\eta_{B'}(z)\bigl|K(-y)-K(z-y) \bigl|d(\nu_j+\nu)(z)d(\nu_j+\nu)(y)$$ is at most $C\bigl(\tfrac{\text{diam}(B')}{N}\bigl)^{\alpha}$ (see Lemma \ref{taildiff}).  On the other hand, for any $N>0$, the weak convergence of $\nu_j$ to $\nu$ guarantees that $T(\psi_N \nu_j )(0)$ converges to $T(\psi_N \nu)(0)$ as $j\rightarrow \infty$, and also that  $\int_{B'}\lambda_j\eta_{B'} T_{\nu_j}(\psi_N)d\nu_j$ converges to $\int_{B'}\eta_{B'} T_{\nu}(\psi_N)d\nu$ as $j\rightarrow \infty$ (see Section \ref{primer}).  This establishes the required convergence.
\end{proof}

\section{The Collapse Lemma}\label{colsec}

This section is devoted to introducing the main technical tool of the paper.  Throughout the section, suppose that $\mu$ is a non-trivial $\Lambda$-good reflectionless measure.

For a unit vector $\e\in \mathbb{C}^{d'}$, and $\eps\in \mathbb{R}$, define
$$E(\e, \eps, r) = \bigl\{ x\in \mathbb{R}^d: \Re[\e\cdot \T1_{\mu,\delta}(1)](x)>\eps \text{ for all }\delta\in (0,r)\bigl\}.
$$

\begin{prop}[The Collapse Lemma] \label{collem} Let $\eps\in (0,\tfrac{1}{2})$ and $\e\in \mathbb{C}^{d'}$, $|e|=1$.  There exists $\beta>0$ (depending on  $s$ and $\alpha$), such that if $\kap \leq \kap(\eps)=c_{12}\eps^{\beta}$, then the following holds:  If $E(\e, \eps, r)\cap B(x_0,2r)$ is $\kap r$-dense in $B(x_0,2r)$, then $\mu(B(x_0,r))=0$.
\end{prop}

First note that by considering the measure $\tfrac{\mu(r\cdot +x_0)}{r^s}$ instead of $\mu$, it suffices to prove the result for $x_0=0$ and $r=1$.  The proof consists of two ideas, which are expressed by the following two lemmas.

\begin{lem}[High density yields geometric decay of measure]\label{decaylem}  Let $\eps\in (0,\tfrac{1}{2})$, $\kap\in (0,\tfrac{1}{4})$, and $t\in (1,2]$.  Suppose that $E(\e,\eps, 1)\cap B(0,t)$ is $\kap $-dense in $B(0,t-\sqrt{\kap})$.  If %\begin{equation}\label{collapsecond1}
$\eps  \geq 2C_6\kap^{\tfrac{\alpha}{2}}$,
%\end{equation}
then $$\mu(B(0,t-\sqrt{\kap}))\leq (1-\lambda)\mu(B(0,t)),$$ with $\lambda=c_8\eps^2$.
\end{lem}

%\textit{Step 1:  Geometric decay.}  The first step in the argument is to show that, provided $\kap$ is small enough,  the ball $B(0, \tfrac{3}{2})$ has $\mu$-measure as small as we wish.  Let $t>1$.  %Consider a finite $\kap$-net $(B_j)_{j\geq 1}$ of the ball $B(0, t-\sqrt{\kap})$.  Let $(\varphi_j)_j$ be a smooth partition of unity satisfying $\sum_{j}^{\infty}\varphi_j \equiv 1$ on $\cup_j B_j$, $\supp(\varphi_j)\subset 2B_j$, and $|\nabla \varphi_j|\leq \tfrac{C}{\kap}.$    Let $\varphi = \sum_j \phi_j$, then $\chi_{B(0,t-\sqrt{\kap})}\leq \varphi \leq \chi_{B(0,t-\sqrt{\kap}+2\kap)}$, and $|\nabla\varphi|\leq \frac{C}{\kap}$.
%Define $\displaystyle \psi_j = \frac{\varphi_j}{\max\bigl(1,\sum_{k\geq1} \varphi_k\bigl)}$, then $|\nabla \psi_j|\leq \tfrac{C}{\kap}$ since the collection $(2B_j)_j$ has a covering number of at most $C$.
%Let $a= \int_{\mathbf{R}^d} \sum_j \psi_j d\mu$, and define $f= \sum_{j\geq 1} \psi_j - a\varphi_0$.  Note that $f$ has $\mu$-mean zero, and $\Lip(f) \leq C(\frac{1}{\kap} + \eta^{-1/10})$.  Hence $$\bigl| \bigl\langle R(\mu), f \rangle \bigl|_{L^2(\mu)} \leq C\bigl(\frac{\eta}{\kap} + \eta^{9/10}\bigl).$$
\begin{proof} For every $x\in B(0, t-\sqrt{\kap})$, there exists $x'\in E(\e, \eps, 1)$ with $|x-x'|<\kap$.  Thus, for any $\delta\in[ \sqrt{\kap},1)$, \begin{equation}\label{kapholder}|\T1_{\mu,\delta}(1)(x) - \T1_{\mu,\delta}(1)(x')|\leq C_6\Bigl(\frac{\kap}{\delta}\Bigl)^{\alpha}\leq C_6\kap^{\tfrac{\alpha}{2}}.\end{equation} As long as $\kap\leq\tfrac{1}{4}$, it follows that $\Re[\e\cdot \T1_{\mu,\sqrt{\kap}}(1)](x)>\eps-C_6\kap^{\tfrac{\alpha}{2}}\geq\tfrac{\eps}{2}$.

%Next, note that the Lipschitz continuity property of $R_{\sqrt{\kap}/2}(\mu)$ guarantees that
%$$\int_{\mathbf{R}^d} \varphi_j \bigl| R_{\sqrt{\kap}/2}(\mu)(x_j) - R_{\sqrt{\kap}/2}(\mu)(z)\bigl| d\mu(z) \leq C\sqrt{\kap} \int_{\mathbf{R}^d} \varphi_j d\mu.
%$$
%Summing the previous inequality over $j$, we arrive at
%$$\sum_{j\geq 1}\int_{\mathbf{R}^d} \psi_j \bigl| R_{\sqrt{\kap}/2}(\mu)(x_j) - R_{\sqrt{\kap}/2}(\mu)(z)\bigl| d\mu(z) \leq C\sqrt{\kap}.
%$$
As a result of this property and the reflectionlessness of $\mu$, we have
\begin{equation}\nonumber\begin{split}\int_{B(0, t-\sqrt{\kap})}\Re[\e\cdot\T1^{\sqrt{\kap}}_{\mu}(1)]d\mu &= -\int_{B(0, t-\sqrt{\kap})}\Re[\e\cdot\T1_{\mu,\sqrt{\kap}}(1)] d\mu \\
&\leq -\frac{\eps}{2} \mu(B(0, t-\sqrt{\kap})).\end{split}\end{equation}
On the other hand,  $\T1^{\sqrt{\kap}}_{\mu}(1) = T^{\sqrt{\kap}}_{\mu}(\chi_{B(0,t)})$ on $B(0, t-\sqrt{\kap})$.  Furthermore, due to the anti-symmetry of the kernel $K$, we may write $$ \int_{B(0, t-\sqrt{\kap})}\!\!\!\Re[\e\cdot T^{\sqrt{\kap}}_{\mu}(\chi_{B(0,t)})]d\mu = \int_{B(0, t-\sqrt{\kap})}\!\!\!\!\Re[\e\cdot T^{\sqrt{\kap}}_{\mu}(\chi_{B(0,t)\backslash B(0, t-\sqrt{\kap}})]d\mu.$$
%$$\int_{\mathbf{R}^d}\sum_j\varphi_j \widetilde{R}^{\sqrt{\kap}/2}(\mu)d\mu < -\Bigl(\eps -C\sqrt{\kap} - \frac{C\eta^{9/10}}{\kap}\Bigl)\mu(B(0,t-\sqrt{\kap})),
%$$
%as long as
But, by the $L^2(\mu)$ boundedness of $T_{\mu}^{\delta}$ (recall that it has operator norm at most $2\Lambda$), the right hand side of this equality is at least $$-2\Lambda \sqrt{\mu(B(0, t-\sqrt{\kap}))\cdot\mu(B(0, t)\backslash B(0, t-\sqrt{\kap}))}. $$
Bringing our estimates together, we see that
\begin{equation}\label{rearrange}\mu(B(0,t-\sqrt{\kap})) \leq \frac{4\Lambda ^2}{4\Lambda ^2 + \tfrac{\eps^2}{4}}\mu(B(0,t)).
\end{equation}
From which is follows that,
%Furthermore, if we suppose that %\begin{equation}\label{collapsecond1}
%\eps  \geq 2C_6\kap^{\tfrac{\alpha}{2}},
%\end{equation}
\begin{equation}\label{decay}\mu(B(0, t-\sqrt{\kap})) \leq (1-\lambda) \mu(B(0,t)),
\end{equation}
with  $\lambda = c_8\eps^2$, for $c_8$ chosen suitably.\end{proof}

The second ingredient is the following lemma.

\begin{lem}[Density increment in a ball of small measure]\label{densinclem} Let $\eps\in (0, \tfrac{1}{2})$, $\kap\in (0, \tfrac{1}{4})$,  $m\in (0, 1)$, and $t\in (1,2]$.  Suppose that $E(\e,\eps,1)\cap B(0,t)$ is $\kap$-dense in $B(0,t-\sqrt{\kap})$, and $\mu(B(0,t))\leq m$.  There exists a constant $C_9$ such that if
$$\eps'=\eps - C_6[\kap^{\tfrac{\alpha}{2}}+\sqrt{m}], \, \kap' = C_9 m^{\tfrac{1}{2d}},\,\text{ and } t'=t-\sqrt{\kap},
$$
then $E(\e, \eps',1)\cap B(0,t')$ is $\kap'$-dense in $B(0,t')$.
\end{lem}

%Given $\eps_{j}>0$, $\kap_j>0$, and $t_j>1$, suppose that the set $E(\e,\eps_j, t_j)\cap B(0,t_j)$ is $\kap_j$-dense in $B(0,t_j)$. Set $\eps_{j+1} = \eps_j - C_6\kap_j^{\tfrac{\alpha}{2}}- C_6\sqrt{\mu(B(0,t_j))}$. Assume that $\eps_{j+1}>0$.

\begin{proof}For any $x\in B(0,t')$, there exists $x'\in E(\e,\eps, t)\cap B(0,t)$ such that  $|x-x'|\leq\kap$.
By writing $\e\cdot\T1_{\mu,\delta}(1)(x) = \e\cdot\T1_{\mu,\delta}(1)(x')+  [\e\cdot\T1_{\mu,\delta}(1)(x) - \e\cdot\T1_{\mu,\delta}(1)(x')]$, we see from (\ref{kapholder}) that \begin{equation}\begin{split}\label{denslowbd} \Re[\e\cdot\T1_{\mu,\delta}(1)](x) %&= \e\cdot\T1_{\mu,\delta}(1)(x')+  \e\cdot\T1_{\mu,\delta}(1)(x) - \e\cdot\T1_{\mu,\delta}(1)(x')  \\
%&
> \eps-C_6\kap^{\tfrac{\alpha}{2}}=\eps'+C_6\sqrt{m},
\end{split}\end{equation}
for any $\delta\in[\sqrt{\kap}, 1)$.

Set $\widetilde{F}_{\delta}(x) = \T1_{\mu, \delta}(1)(x) - \T1_{\mu, \sqrt{\kap}}(1)(x)$.  From (\ref{denslowbd}), we infer that if $\Re[\textbf{e}\cdot\T1_{\mu,\delta}(1)](x) <\eps'$ for some $x\in B(0,t')$ and $\delta\in (0,1)$, then $\delta<\sqrt{\kap}$ and $|\widetilde{F}_{\delta}(x)|>C_6\sqrt{m}$ (the second condition follows since $\Re[\e\cdot\T1_{\mu,\sqrt{\kap}}(1)](x)>\eps'+C_{6}\sqrt{m}$, and certainly $\Re[\e\cdot\T1_{\mu,\delta}(1)](x)\geq \Re[\e\cdot\T1_{\mu,\sqrt{\kap}}(1)](x)-|\widetilde{F}_{\delta}(x)|$).

For $x\in B(0, t')$, $\widetilde{F}_{\delta}(x)= \int_{B(0,t)}[K_{\delta}(x-y)-K_{\sqrt{\kap}}(x-y)]dm_d(y)$.  Whence,
\begin{equation}\begin{split}\nonumber\int_{B(0, t - \sqrt{\kap})}\sup_{\delta\in (0,\sqrt{\kap})}|\widetilde{F}_{\delta}(y)|dm_d(y)&\leq 2\int_{B(0, t - \sqrt{\kap})}\int_{|y-z|<\sqrt{\kap}}\frac{d\mu(z)}{|y-z|^s}dm_d(y)\\
&\leq 2C_1\mu(B(0, t)) \kap^{\tfrac{d-s}{2}}\leq 2C_1m.
\end{split}\end{equation}
Consequently, Chebyshev's inequality yields that
\begin{equation}\begin{split}\nonumber m_d\bigl(\bigl\{x\in B(0, t')\,: \sup_{\delta \in (0,\sqrt{\kap})}|\widetilde{F}_{\delta}(x)|>C_6\sqrt{m}\bigl\}\bigl) \leq \frac{2C_1}{C_6}\sqrt{m}.\end{split}\end{equation}
%Now set $t_{j+1} = t_j - \sqrt{\kap_j}$, and $\kap_{j+1} = C_{8}\mu(B(0,t_j))^{\tfrac{1}{2d}}$, where
Now, fix $C_9  \geq \bigl(\tfrac{4C_1}{\omega_d C_6}\bigl)^{\tfrac{1}{d}}$, where $\omega_d$ denotes the $d$-dimensional volume of the unit ball. Then the set $E(\e,\eps', t')\cap B(0,t')$ is $\kap'$-dense in $B(0, t')$ if $\kap'<\tfrac{1}{4}$.  Indeed, $m_d(B(0,t')\backslash E(\e, \eps', t'))<\omega_d\bigl(\tfrac{\kap'}{2}\bigl)^d$.  But, if for any $x\in B(0,t')$, the closest point of $E(\e, \eps', t')\cap B(0,t')$ is at a distance greater than $\kap'$, then there is a ball of radius $\tfrac{\kap'}{2}$ that is contained in $B(0,t')$ but disjoint from $E(\e,\eps',t')$.  The existence of this ball is in contradiction with the measure estimate.  If $\kap'\geq \tfrac{1}{4}$, then $\kap'\geq \kap$, so there is nothing to prove.
\end{proof}

We now combine Lemmas \ref{decaylem} and \ref{densinclem} to prove the Collapse Lemma.

\begin{proof}[Proof of Proposition \ref{collem}]  Fix $m_0>0$ to be chosen later, and suppose that $\mu(B(0,\tfrac{3}{2}))\leq m_0$.  Set $t_0 = \tfrac{3}{2}$, $\kap_0 =\kap$, and $\eps_0=\eps$.  Then $E(\e, \eps_0, 1)\cap B(0,t_0)$ is $\kap_0$-dense in $B(0,t_0-\sqrt{\kap_0})$ by the hypotheses of  Proposition \ref{collem}.  For $j\geq 1$, set
$$\eps_{j} =\eps_0 -\sum_{\ell=0}^{j-1}C_6\bigl[\kap_{\ell}^{\tfrac{\alpha}{2}}+\sqrt{m_{\ell}}\bigl],\; \kap_j = C_{9}m_{j-1}^{\tfrac{1}{2d}},\; t_j = t_0 - \sum_{\ell=0}^{j-1}\sqrt{\kap_{\ell}},
$$
and $m_j = (1-\tfrac{\lambda}{4})m_{j-1}$, with $\lambda=c_8\eps^2$ as in Lemma \ref{decaylem}.

Suppose that for some $j\geq 0$, $E(\e, \eps_j, 1)\cap B(0,t_j)$ is $\kap_j$-dense in $B(0,t_j-\sqrt{\kap_j})$, and also that $\mu(B(0,t_j))\leq m_j$.   If \begin{equation}\label{epsjkapj}\eps_j \geq \frac{\eps}{2}, \;\; 2C_6\kap_j^{\tfrac{\alpha}{2}}\leq \frac{\eps}{2},\; \text{ and }t_j>1,\end{equation}then $2C_6\kap_j^{\tfrac{\alpha}{2}}\leq \eps_j$, and Lemma \ref{decaylem} yields $\mu(B(0,t_{j+1}))\leq (1-c_8\eps_j^2)m_j$. But since $\eps_j\geq \tfrac{\eps}{2}$, we have $c_8\eps_j^2 \geq \tfrac{\lambda}{4}$, and so $\mu(B(0,t_{j+1}))\leq m_{j+1}$.  %In particular, this is satisfied whenever

On the other hand, Lemma \ref{densinclem} ensures that $E(\e, \eps_{j+1},1)\cap B(0,t_{j+1})$ is $\kap_{j+1}$-dense in $B(0,t_{j+1})$.

Bringing these two observations together, we see that if (\ref{epsjkapj}) holds for each $j\geq 0$, then $$\mu(B(0,t_{j}))\leq \Bigl(1-\frac{\lambda}{4}\Bigl)^jm_0 \text{ for every }j\geq 0,$$
and so $\mu(B(0,1))=0$, which is the desired conclusion of the Collapse Lemma.

We shall now make a choice of parameters to ensure that (\ref{epsjkapj}) is valid.  Our requirement that $\eps_j \geq \tfrac{\eps}{2}$ and  $2C_6\kap_j^{\tfrac{\alpha}{2}}\leq\tfrac{\eps}{2}$ for every $j$ will be satisfied if
$$C_6\kap^{\tfrac{\alpha}{2}}+\sum_{\ell=0}^{\infty}C_6\Bigl[C_9^{\tfrac{\alpha}{2}}\Bigl(1-\frac{\lambda}{4}\Bigl)^{\tfrac{\alpha \ell}{4d}}m_0^{\tfrac{\alpha}{4d}}+\Bigl(1-\frac{\lambda}{4}\Bigl)^{\tfrac{\ell}{2}}\sqrt{m_0}\Bigl]<\frac{\eps}{2} \text{ and }C_9m_0^{\tfrac{1}{2d}}<\frac{\eps}{2}.
$$
While $t_j>1$ for all $j\geq 1$ if $$\sum_{\ell=0}^{\infty}\sqrt{C_9}\Bigl(1-\frac{\lambda}{4}\Bigl)^{\tfrac{\ell}{4d}}m_0^{\tfrac{1}{4d}}<\frac{1}{2}.
$$

Notice that $\sum_{\ell=0}^{\infty}\bigl(1-\tfrac{\lambda}{4}\bigl)^{\tfrac{\alpha \ell}{4d}}\leq \tfrac{C}{\lambda}\leq \tfrac{C}{\eps^2}$.  Therefore, if we choose $m_0 = c_{10}\eps^{\gamma}$ for suitable constants $c_{10}>0$ and $\gamma = \gamma(d,s,\alpha)>0$, then the inequalities comprising (\ref{epsjkapj}) are satisfied provided that $\kap<\bigl(\tfrac{\eps}{4C_6}\bigl)^{\tfrac{2}{\alpha}}$.

%We shall now iterate this argument.  Suppose that (\ref{collapsecond1}) holds, then for every $k\in \mathbb{N}$ satisfying $k\sqrt{\kap}<1$, \begin{equation}\label{iterateddecay1}\begin{split}\mu(B(0,2 - k\sqrt{\kap}))&\leq (1-\lambda) \mu(B(0,2-(k-1)\sqrt{\kap})) \\
%&\leq (1-\lambda)^k\mu(B(0,2)).\end{split}\end{equation}

It remains to ensure that $\mu(B(0,t_0)) = \mu(B(0,\tfrac{3}{2}))\leq m_0$.  To do this, we shall require an additional restriction upon the density parameter $\kap$.  Fix $N\in \mathbb{N}$.  %As long as $N\sqrt{\kap}<\tfrac{1}{2}$,
We may repeatedly apply Lemma \ref{decaylem} to yield
\begin{equation}\nonumber\begin{split} \mu(B(0, 2-N\sqrt{\kap})) \leq (1-\lambda)^N\mu(B(0,2))\leq (1-\lambda)^N\Lambda2^s.
\end{split}\end{equation}
%Fix $m_0>0$ to be chosen later.  Choose a positive integer $N$ large enough to ensure that
If $N\sqrt{\kap}<\tfrac{1}{2}$, then $\mu(B(0,\tfrac{3}{2}))\leq (1-\lambda)^N\Lambda2^s.$  Thus, we require that $(1-\lambda)^N \leq \tfrac{m_0}{\Lambda 2^s}$.
 %\begin{equation}\label{lambdaNcond}(1-\lambda)^N \leq \frac{m_0}{\Lambda 2^s}.\end{equation}
This condition (which is the sole condition on $N$ that is independent of $\kap$) dictates our choice of $N$ as  $N= \lfloor C_{11}\tfrac{\log\tfrac{1}{\eps}}{\eps^2}\rfloor+1.$  %The inequality (\ref{Ncond1}) now yields a final restriction on $\kap$, namely that $\kap \leq c_{12}\tfrac{\eps^4}{(\log \eps)^2}$.
All that is left is to choose $\kap(\eps)$.  The two assumptions we need to satisfy are
$$\kap(\eps)<\Bigl(\frac{\eps}{4C_9}\Bigl)^{\tfrac{2}{\alpha}}, \text{ and }\kap(\eps)<\frac{1}{4N^2}<\frac{\eps^4}{4C_{11}^2\log^2\tfrac{1}{\eps}}.
$$
So  we can choose $\kap(\eps)=c_{12}\eps^{\beta}$, for suitable $c_{12}>0$ and $\beta=\beta(s,\alpha)>0$.\end{proof}

\subsection{Consequences of the Collapse Lemma}

The remainder of the section is devoted to consequences of the Collapse lemma.  We begin with a simple alternative:

\begin{lem}\label{colalt}  For each $\eps\in(0,\tfrac{1}{2})$, there exist $M=M(\eps)>0$ and $\tau = \tau(\eps)>0$, such that whenever $|\T1_{\mu, Mr}(1)(x)|>\eps$, one of the following two statements must hold:

(i) $\mu(B(x,2Mr)\geq \tau r^s$, or

(ii)  $\mu(B(x,r))=0$.
\end{lem}

\begin{proof}  We may assume that $x=0$ and $r=1$.  Fix $\tau>0$, and $M>4$.  Suppose that $\mu(B(0,2M))\leq\tau$.  Let $\widetilde{F}_{\delta}(x) = \T1_{\mu, \delta}(1) - \T1_{\mu, M}(1)$.  Then $\widetilde{F}_{\delta}= \int_{B(0,2M)}[K_{\delta}(\cdot-y)-K_M(\cdot-y)]d\mu(y)$ on $B(0,2)$.  Thus, by Lemma \ref{locl1},
$$\int_{B(0,2)} \sup_{\delta>0} \bigl|\widetilde{F}_{\delta}(y)\bigl| dm_d(y) \leq 2\cdot 2^{d-s}C_1\tau.
$$
Consequently, the Chebyshev inequality ensures that the set $E = \bigl\{y\in B(0,2): \sup_{\delta>0} |\widetilde{F}_{\delta}(y)|<\tfrac{\eps}{4}\bigl\}$ is $C_{13}\bigl(\tfrac{\tau}{\eps}\bigl)^{\tfrac{1}{d}}$-dense in $B(0,2)$ (cf. the proof of Lemma \ref{densinclem}).

Set $\e$ to be the unit vector satisfying $\langle \T1_{\mu, M}(1)(0), \e \rangle = |\T1_{\mu, M}(1)(0)|$.  Suppose $y\in E$, and $\delta\in (0,1)$.  Write
$$\T1_{\mu,\delta}(1)(y) = \T1_{\mu,M}(1)(0) + \widetilde{F}_{\delta}(y) + [\T1_{\mu,M}(1)(y)- \T1_{\mu,M}(1)(0)].
$$
Since $|\T1_{\mu,M}(1)(y)- \T1_{\mu,M}(1)(0)|\leq\tfrac{2C_6}{M^{\alpha}}$, we infer from the above equality that
$\Re [\e\cdot\T1_{\mu, \delta}(1)](y)>\tfrac{3\eps}{4} - \frac{2C_6}{M^{\alpha}}.$  This quantity is at least $\tfrac{\eps}{2}$ if $M\geq M(\eps)= (\eps/8C_6)^{-\tfrac{1}{\alpha}}$.  Now fix $\kap=\kap(\tfrac{\eps}{2}\bigl)$ as in the Collapse lemma.    If $C_{13}\bigl(\tfrac{\tau}{\eps}\bigl)^{1/d}\leq \kap$, then $\mu(B(0,1)=0$.  Thus, fixing $\tau=c_{14}\eps\kap^d$ for a suitable constant $c_{14}>0$ establishes the alternative.
\end{proof}

This lemma yields several useful Corollaries.  The first one shall prove useful in establishing Proposition \ref{onlytrivmeas}.

\begin{cor}\label{collapsecor1}  For each $\eps\in(0,\tfrac{1}{2})$, there exist $M'=M'(\eps)>0$ and $\tau=\tau'(\eps)>0$, such that if $|\T1_{\mu}(1)(x)|>\eps$, and $\dist(x, \supp(\mu))=r$, then $\mu(B(x,M'r))\geq \tau' r^s.$\end{cor}

\begin{proof}  Without loss of generality, we may assume that $r=\tfrac{1}{2}$ and $x=0$.   Set $M=M\bigl(\tfrac{\eps}{2}\bigl)$ as in Lemma \ref{colalt}, and fix $M'=4M$.  Let $\sigma>0$, and suppose that $\mu(B(0, M'))\leq \sigma$.
By a trivial absolute value estimate, $|\T1_{\mu, M}(1)(0)|>\eps-\int_{B(0,M)\backslash B(0,\tfrac{1}{2})}\tfrac{1}{|y|^s}d\mu(y)>\eps-C\sigma$.  So $|\T1_{\mu, M}(1)(0)|>\tfrac{\eps}{2}$ if $\sigma= c\eps$ for a sufficiently small constant $c>0$.  Under this condition on $\sigma$, the assumptions of Lemma \ref{colalt} are satisfied with $\eps$ replaced by $\tfrac{\eps}{2}$.  By hypothesis $\mu(B(0,1))>0$, so $\mu(B(0,2M))>\tau$, where $\tau=\tau\bigl(\tfrac{\eps}{2}\bigl)$ is given by Lemma \ref{colalt}.  Setting $\tau' = \min\bigl[\sigma, \tau\bigl]$ completes the proof.
\end{proof}

The next corollary concerns the values of $\T1_{\mu}(1)$ $m_d$-almost everywhere on the support of $\mu$.

\begin{cor}\label{lebcor}  $\T1_{\mu}(1)(x)=0$ for $m_d$-almost every $x\in \supp(\mu)$.
\end{cor}

\begin{proof}  By standard measure theory, the limit $D(x)=\lim_{r\rightarrow 0} \tfrac{\mu(B(x,r))}{r^d}$ exists and is finite for $m_d$-almost every $x\in \mathbb{R}^d$.  It therefore suffices to prove that if $|\T1_{\mu}(1)(x)|>2\eps$ for some $\eps>0$, and $D(x)$ exists and is finite, then $x\not\in \supp(\mu)$.  Set $M=M(\eps)$, and $\tau=\tau(\eps)$, as in Lemma \ref{colalt}.  If $D(x)<\infty$, then $\mu(B(x,r))\leq (D(x)+1)r^d$ for all sufficiently small $r$.   But, then provided that $r$ is sufficiently small, we have that both $|\T1_{\mu, Mr}(1)|>\eps$ and $\mu(B(x,2Mr))\leq \tau r^s.$  From Lemma \ref{colalt}, we infer that $\mu(B(x,2r))=0$.   So $x\not\in \supp(\mu)$.
\end{proof}

The final result of this section is a porosity property in balls where $\T1_{\mu}(1)$ is large on average.  This will serve as the primary tool in proving that the support of a reflectionless measure for the Riesz transform is nowhere dense.

\begin{lem}\label{intpor}  For each $\eps>0$, there exists $\lambda=\lambda(\eps)>0$, such that if $\int_{B(x,r)}|\T1_{\mu}(1)(y)| dm_d(y)>\eps m_d(B(x,\eps))$, then there is a ball $B'\subset B(x,r)$ of radius $\lambda r$ with $\mu(B')=0$.
\end{lem}

\begin{proof}  We may suppose that $x=0$ and $r=1$.  For $\Delta\in (0,1)$, set $F_{\Delta}(z) = \sup_{\delta\in (0,\Delta]}|\T1_{\mu}^{\delta}(1)(z)|$.  Then
$$\int_{B(0,2)}|F_{\Delta}|dm_d\leq 2\int_{B(0,2)}\int_{B(z,\Delta)}\frac{d\mu(y)}{|z-y|^s}dm_d(z),
$$
which by Lemma \ref{locl1} is at most $2C_1\Delta^{d-s}\mu(B(0,3))\leq 2C_1\Delta^{d-s}\Lambda 3^s$.  Thus,
$$\frac{1}{m_d(B(0,1))}\int_{B(0,1)}|\T1_{\mu, \Delta}(1)|dm_d>\eps - \frac{2C_1\Delta^{d-s}\Lambda 3^s}{\omega_d}>\frac{\eps}{2},
$$
provided that $\Delta<c_{15}\eps^{\tfrac{1}{d-s}}$ (here $\omega_d$ is the volume of the $d$-dimensional unit ball).

Now fix $\gamma\in (0,\Delta)$, and let $(x_j)_j$ be a maximal $\gamma$-separated set in $B(0,1)$.  Set $B_j = B(x_j, \gamma)$.  Then the balls $B_j$ form a cover of $B(0,1)$, are contained in $B(0,2)$, and have bounded covering number (at most $C_{16}$ balls may intersect at any point of $\mathbb{R}^d$).

Set $\e_j$ to be the unit vector with $ \e_j\cdot \T1_{\mu, \Delta}(1)(x_j)\rangle =|\T1_{\mu, \Delta}(x_j)|$.  Denote
$$U_j = \inf_{B_j}\Re\bigl[\e_j \cdot \T1_{\mu,\Delta}(1)\bigl],\text{ and }V_j = \frac{1}{m_d(B_j)}\int_{B_j}F_{\Delta}dm_d.
$$
For $T>0$, set $\mathcal{F}_1 = \{j: U_j<TV_j\}$.  Also, set $\mathcal{F}_2 = \{j: U_j<\tfrac{\eps}{C_{16}16^d}\}$.  Then,
\begin{equation}\begin{split}\nonumber\sum_{j\in \mathcal{F}_1\cup\mathcal{F}_2} \frac{m_d(B_j)}{m_d(B(0,1))}U_j &\leq \frac{TC_{16}}{\omega_d}\int_{B(0,2)}F_{\Delta}dm_d+\frac{\eps}{C_{16} 16^d}\frac{C_{16}\omega_d 2^d}{\omega_d}\\
&\leq C_{17}\Delta^{d-s}T+\frac{\eps}{8}.
\end{split}\end{equation}
This quantity is at most $\tfrac{\eps}{4}$ if $C_{17}\Delta^{d-s}T<\tfrac{\eps}{8}$.
On the other hand, $U_j \geq \tfrac{1}{m_d(B_j)}\int_{B_j}|\T1_{\mu, \Delta}(1)|dm_d - 2C_{6}\bigl(\tfrac{\gamma}{\Delta}\bigl)^{\alpha}.$  Thus
$$\sum_j \frac{m_d(B_j)}{m_d(B(0,1))}U_j\geq \frac{\eps}{2}-2^dC_{16}2C_{6}\Bigl(\frac{\gamma}{\Delta}\Bigl)^{\alpha}>\frac{\eps}{4},$$
if $\gamma\leq c_{18}\eps^{\tfrac{1}{\alpha}}\Delta$.

%Bringing our estimates together, we find a constant $C_{18}$ such that if
%$C_{18}\Delta^{d-s}T <\eps,$ then
%$$\sum_{j\in \mathcal{F}_1\cup\mathcal{F}_2} \frac{m_d(B_j)}{m_d(B(0,1))}U_j<\frac{\eps}{4}.
%$$
Accordingly, there exists a ball $B_j$ such that $U_j\geq TV_j$ and $U_j\geq \tfrac{\eps}{C_{16}16^d}$.

Since $V_j\leq\tfrac{U_j}{T}$, Chebyshev's inequality ensures that
$$m_d\bigl(\bigl\{x\in B_j: |F_{\Delta}(x)|\geq\frac{U_j}{2}\bigl\}\bigl)\leq \frac{2}{T}m_d(B_j).
$$
Now fix $\kap = \kap\bigl(\tfrac{\eps}{2C_{16}16^d}\bigl)$ be as in the Collapse lemma (Proposition \ref{collem}).   Choose $T= \tfrac{4^d}{\kap^d}$.  Then the set $E=\{x\in B_j: |F_{\Delta}|<\tfrac{U_j}{2}\}$ is $\kap \gamma$-dense in $B_j$ (see the proof of Lemma \ref{densinclem}).  But for each $x\in E$, $\Re[\e_j\cdot \T1_{\mu, \delta}(1)](x)>\tfrac{\eps}{2C_{16}16^d}$ for any $\delta\in (0, \Delta]$.  Thus the Collapse lemma yields that $\mu(\tfrac{1}{2}B_j)=0$.

Let us now choose the parameters $\Delta$ and $\gamma$.  There are two conditions on $\Delta$ independent of $\gamma$, namely
$$C_{17}\Delta^{d-s}T <\frac{\eps}{8}, \text{ and }\Delta<c_{15}\eps^{\tfrac{1}{d-s}}.
$$
To satisfy these inequalities, we may choose $\Delta = c_{19}\eps^{\beta'}$, for some $\beta'>0$ depending on $s$, $d$, and $\alpha$.  It remains to choose $\gamma$ subject to the inequality $\gamma\leq c_{18}\eps^{\tfrac{1}{\alpha}}\Delta$.  Thus we can fix $\gamma = c_{18}c_{19}\eps^{\tfrac{1}{\alpha}}\eps^{\beta'}$.
Since the centre of $\tfrac{1}{2}B_j$ lies in $B(0,1)$, there is a ball of radius $\lambda := \tfrac{\gamma}{4}$,  disjoint from $\supp(\mu)$, that is contained in $B(0,1)$.
\end{proof}

\section{The Riesz transform}

From here on in, we focus on the simplest, and most interesting $s$-dimensional CZO, the $s$-Riesz transform.  This is the choice of kernel $K(x)=\tfrac{x}{|x|^{s+1}}$ for $x\in \mathbb{R}^d$ (so the Riesz transform is $\mathbb{R}^d$-valued).  We will write $R_{\mu}$ instead of $T_{\mu}$, $\R1_{\mu}(1)$ instead of $\T1_{\mu}(1)$, and so on.

%The reason why we are able to say more about the Riesz transform is for the same reason why the Riesz transform is of most interest:  it is closely linked to the fractional Laplacian operator of order $\tfrac{d+1-s}{2}$.  For instance, if one can show that property ($*$) holds for the $s$-dimensional Riesz transform, then it follows that sets of finite $\mathcal{H}^s$-measure are removable for Lipschitz continuous $\tfrac{d+1-s}{2}$-harmonic functions.  For more information about the case $s=d-1$, the reader may consult the recent paper \cite{NTV2}.

\subsection{An extremal problem: the proof of Proposition \ref{onlytrivmeas}}

%With the main objects in place, we turn to the principal results of this paper.

In this section we prove Proposition \ref{onlytrivmeas}.  It will follow from the following proposition:

\begin{prop}\label{genprop}  Let $s\in (0,d)$.  Suppose that $\mu$ is a non-trivial $\Lambda$-good reflectionless measure (for the $s$-Riesz transform).  Then $\R1_{\mu}(1)\in L^{\infty}(m_d)$.  Furthermore, there exists a $\Lambda$-good reflectionless measure $\mu^{\star}$, with $\dist(0,\supp(\mu^{\star}))= 1$, and
$$|\R1_{\mu^{\star}}(1)(0)| = \|\R1_{\mu^{\star}}(1)\|_{L^{\infty}(m_d)}.
$$
\end{prop}

%In the case when $s\in (d-1,d)$, we can use Proposition \ref{genprop} to say more.

%\begin{prop}\label{sdprop}  If $s\in (d-1,d)$, then there is no nontrivial reflectionless measure.
%\end{prop}

%As we shall see, Theorem \ref{mainthm} follows from Proposition \ref{sdprop} in a quite straightforward fashion.

%\section{The proof of Proposition \ref{sdprop} (assuming Proposition \ref{genprop})}

Before we prove this result, let's see how Proposition \ref{onlytrivmeas} follows from it.  We shall require two lemmas.  Both of them are essentially known, and so the proofs are relegated to an appendix.

%The first lemma encodes the action of the Riesz transform on the Fourier side.

\begin{lem}\label{philem}  Suppose that $s\in (0,d)$, and $\mu$ is a $\Lambda $-good measure.  Then there exists a constant $C_{20}>0$ such that for any ball $B(x,r)$ and $\Gamma\in \mathbb{R}^d$,
$$\mu(B(x,r)) \leq C_{20}r^d\int_{\mathbb{R}^d} \frac{|\R1_{\mu}(1)(z) - \Gamma|}{(r+|z|)^{2d-s}} dm_d(z).
$$
\end{lem}

The next lemma accounts for the restriction to $s\in (d-1,d)$.

\begin{lem}\label{sharmon} Suppose that $s\in (d-1,d)$, and $\mu$ is a $\Lambda $-good measure, with $0\not\in \supp(\mu)$.  Then
\begin{equation}\begin{split}\nonumber P.V. \int_{\mathbb{R}^d} &\frac{\R1_{\mu}(1)(0)  - \R1_{\mu}(1)(x)}{|x|^{2d+1-s}} dm_d(x)\\
&=\lim_{\delta\rightarrow 0}\int_{\mathbb{R}^d\backslash B(0,\delta)}\frac{\R1_{\mu}(1)(0) - \R1_{\mu}(1)(x)}{|x|^{2d+1-s}} dm_d(x)=0.
\end{split}\end{equation}
\end{lem}

%With these two results in hand, we see how Proposition \ref{sdprop} follows from Proposition \ref{genprop}.

\begin{proof}[Proof of Proposition \ref{onlytrivmeas}]  Suppose that $\mu$ is a nontrivial reflectionless measure.  Consider the measure $\mu^{\star}$ provided by Proposition \ref{genprop}.  Note that Lemma \ref{sharmon} implies that $\R1_{\mu^{\star}}(1)$ is constant $m_d$-almost everywhere.  But then Lemma \ref{philem}, applied with $\Gamma = \R1_{\mu}(1)(0)$, yields that $\mu(B(x,r))=0$ for every $x\in \mathbb{R}^d$ and $r>0$.  This is a contradiction.
\end{proof}

%\section{An extremal problem: the proof of Proposition \ref{genprop}}

We now set up an extremal problem whose solution will provide the measure $\mu^{\star}$ whose existence is claimed in the statement of Proposition \ref{genprop}.  Suppose that there exists a non-trivial $\Lambda$-good reflectionless measure $\mu$.

Define $\mathcal{F}$ to be the set of non-trivial $\Lambda$-good reflectionless measures $\mu$.  Set $\mathcal{Q} = \sup\{|\R1_{\mu}(1)(0)|: \mu\in \mathcal{F}\text{ with } \dist(0, \supp(\mu))=1\}$.

\begin{cla}  $\mathcal{Q}>0$.
\end{cla}

\begin{proof} Lemma \ref{philem} (applied with $\Gamma=0$) yields that if $|\R1_{\mu}(1)|=0$ $m_d$-almost everywhere in $\mathbb{R}^d$, then $\mu=0$. If $\mu\in \mathcal{F}$, then $ |\R1_{\mu}(1)|=0$ $m_d$-almost everywhere on $\supp(\mu)$ (Corollary \ref{lebcor}), and so there must be a point $z \not\in \supp(\mu)$ with $|\R1_{\mu}(1)(z)|>0$.  Set $p=\dist(z, \supp(\mu))$.  Consider the measure $\tilde\mu(\cdot) = \tfrac{\mu(p \cdot+ z)}{p^s}$.  Then $\tilde\mu \in \mathcal{F}$, $\dist(0, \supp(\mu))=1$, and $|\R1_{\tilde\mu}(1)(0)|=|\R1_{\mu}(1)(z)| >0$.
\end{proof}

%\begin{rem}  The same argument shows that if  $\mu$ is a non-trivial reflectionless measure, then $\supp(\mu)\neq \mathbb{R}^d$.
%\end{rem}

\begin{cla}  $\mathcal{Q}<+\infty$.
\end{cla}

\begin{proof}  This follows immediately from the Cotlar lemma (see Corollary \ref{reflinfbd}).
\end{proof}

%We now record a convergence result.

\begin{cla}\label{extremecla}  There exists $\mu^{\star}\in \mathcal{F}$ with $\dist(0, \supp(\mu))= 1$, such that $\R1_{\mu^{\star}}(1)(0)=\mathcal{Q}$.
\end{cla}

\begin{proof}
For each $j\in \mathbb{N}$, choose $\mu_j\in \mathcal{F}$ with $\dist(0, \supp(\mu_j))=1$, satisfying $|\R1_{\mu_j}(1)(0)|\geq \mathcal{Q}(1-2^{-j-1})$.
Then, by Corollary \ref{collapsecor1}, there exists $M'=M'(\mathcal{Q})$ such that $\mu_j(B(0, M'))\geq c(\mathcal{Q})$ for each $j$.  We may pass to a subsequence that converges to a $\Lambda$-good reflectionless measure $\mu^{\star}$ (Corollary \ref{reflweaklim}).  From standard weak semi-continuity properties of the weak limit, we have that $\dist(0, \supp(\mu^{\star}))\geq 1$, and $\mu^{\star}(\overline{B(0,M')})\geq c(\mathcal{Q})$.  Thus Lemma \ref{lebconv} is applicable, and yields a measure $\mu\in \mathcal{F}$ with $\dist(0,\supp(\mu))\geq 1$ and $\R1_{\mu}(1)(0)=\mathcal{Q}$.  Fix $p=\dist(0,\supp(\mu))$.  Setting $\mu^{\star}(\,\cdot\,) = \tfrac{\mu(p\,\cdot\,)}{p^s}$ yields the claim.
\end{proof}

\begin{proof}[The proof of Proposition \ref{genprop}]  Corollary \ref{reflinfbd} yields that $\|\R1_{\mu}(1)\|_{L^{\infty}(m_d)}<\infty$.  Consider the measure $\mu^{\star}$ constructed in Claim \ref{extremecla}, and suppose that $|\R1_{\mu^{\star}}(1)(0)|<\|\R1_{\mu}(1)\|_{L^{\infty}(m_d)}$.  As a result of Corollary \ref{lebcor}, there exists $x\not\in \supp(\mu)$ with $|\R1_{\mu^{\star}}(1)(x)|>|\R1_{\mu^{\star}}(1)(0)|$.  But now set $p=\dist(x, \supp(\mu^{\star}))$.  Consider $\tilde\mu(\,\cdot\,) = \tfrac{\mu(p\,\cdot\,+x)}{p^s}$.  Then $\tilde\mu\in \mathcal{F}$, and $\mathcal{Q}<|\R1_{\mu^{\star}}(1)(x)| = |\R1_{\tilde\mu}(1)(0)|$.  This is absurd.
\end{proof}

%\begin{proof}[Proof of Theorem \ref{mainthm}]
%Propositions \ref{sdprop} and \ref{lowbdtheta} combine to yield that property $(*_{\mu})$ holds for any non-trivial $\Lambda $-good measure $\mu$.  Now Proposition \ref{startriv} completes the proof.
%\end{proof}

\subsection{Weak porosity: the proof of Theorem \ref{nodense}}

Having proved that non-trivial reflectionless measures for the $s$-Riesz transform fail to exist if $s\in (d-1,d)$, we move onto a studying them for $s\leq d-1$.  Theorem \ref{nodense} is an immediate consequence of the following proposition.

%\begin{prop}\label{Rieszporous}  Suppose that $\mu$ is a reflectionless measure for the $s$-dimensional Riesz transform, with $s\in (0,d)$  For each $%\eps>0$ there is a constant $\lambda=\lambda(\eps)>0$ such that if
%\begin{itemize}
%\item $s\neq d-1$ and $\mu(B(x,r))>\eps r^s$, or
%\item  $s=d-1$ and $\mu(B(x,\tfrac{r}{2}))>\eps r^s$,
%\end{itemize}then there is a ball $B'\subset B(x,r)$ of radius $\lambda r$ that does not intersect $\supp(\mu)$.
%\end{prop}

\begin{prop}\label{Rieszporous}  Suppose that $\mu$ is a reflectionless measure for the $s$-dimensional Riesz transform, with $s\in (0,d-1]$.  For each $\eps>0$ there is a constant $\lambda=\lambda(\eps)>0$ such that if $\mu(B(x,r))>\eps r^s$, then there is a ball $B'\subset B(x,3r)$ of radius $\lambda r$ that does not intersect $\supp(\mu)$.
\end{prop}

Taking into account Lemma \ref{intpor}, Proposition \ref{Rieszporous} will follow immediately from the following result.

%\begin{lem}  Let $s\in (0,d-1]$.  There is a constant $c>0$, such that if $\mu(B(x,r))\geq \eps r^s$, then $\dashint_{B(x,r)}|\R1_{\mu}(1)|dm_d\geq c\eps$ if $s<d-1$, and $\dashint_{B(x,2r)}|\R1_{\mu}(1)| dm_d>c\eps$ if $s=d-1$.
%\end{lem}

\begin{lem}  Let $s\in (0,d-1]$.  There is a constant $c_{21}>0$, such that if $\mu(B(x,r))\geq \eps r^s$, then  $\int_{B(x,3r)}|\R1_{\mu}(1)| dm_d>c_{21}\eps m_d(B(x,3r))$.
\end{lem}

\begin{proof}  We may assume that $x=0$ and $r=1$.  Let $\psi_{\tfrac{1}{2}}$ be a non-negative bump function supported in $B(0,\tfrac{1}{2})$, with $\int_{\mathbb{R}^d} \psi_{\tfrac{1}{2}}dm_d=1$.  Then $(\psi_{1/2}*\mu)(B(0,2))\geq c\eps$.  There is a positive constant $b=b(s)$ such that
$$\text{div}(\psi_{\tfrac{1}{2}}*\R1_{\mu}(1))(x) = \begin{cases} b(\psi_{\tfrac{1}{2}}*\mu)(x) \text{ if }s=d-1\\ b\int_{\mathbb{R}^d} \tfrac{d(\psi_{1/2}*\mu)(y)}{|x-y|^{s+1}} \text{ if }s<d-1.\end{cases}
$$
On the other hand, if $\varphi\in C^{\infty}_0(\mathbb{R}^d)$, then
$$\int_{\mathbb{R}^d}[\psi_{1/2}*\R1_{\mu}(1)]\cdot \nabla \varphi dm_d = \int_{\mathbb{R}^d} \text{div}(\psi_{1/2}*\R1_{\mu}(1)) \varphi dm_d.$$
Choose $\varphi$ to be nonnegative, with bounded gradient, and satisfying $\varphi \equiv 1$ on $B(0,2)$, $\supp(\varphi)\subset B(0, \tfrac{5}{2})$.  Then
\begin{equation}\begin{split}\nonumber C\int_{B(0,3)}|\R1_{\mu}(1)| dm_d &\geq \Bigl|\int_{\mathbb{R}^d}[\psi_{1/2}*\R1_{\mu}(1)]\cdot \nabla \varphi dm_d\Bigl|\\&\geq \int_{B(0,2)}\text{div}(\psi_{1/2}*\R1_{\mu}(1))  dm_d\geq c\eps,
\end{split}\end{equation}
as required.
%If $s<d-1$, chose $\varphi$ nonnegative, $\varphi\equiv 1$ on $B(0,\tfrac{1}{4})$, supported in $B(0,\tfrac{1}{2})$, with bounded gradient.  If $s=d-1$, instead choose $\varphi$ nonnegative, equal to $1$ on $B(0,1)$, supported in $B(0, \tfrac{3}{2})$, and with bounded gradient.  The crudest of estimates now yields the desired conclusion.
%Choosing $\varphi$ to be a suitable nonnegative bump function (localized to $B(0,\tfrac{1}{2})$ if $s<d-1$, and to $B(0, \tfrac{3}{2})$ if $s=d-1$) completes the proof.
\end{proof}

\appendix

\section{The operator $T_{\mu}$}\label{T1append}
First suppose that $\mu$ is merely $\Lambda $-nice.  Then, for $f, g\in \Lip_0(\mathbb{R}^d)$ ($f$ scalar valued, $g$ vector valued), we may use the anti-symmetry of the function $K_{\delta}$ to write
$$ I_{\mu,\delta}(f,g) :=\langle T_{\mu, \delta}(f), g\rangle_{\mu} = \int_{\mathbb{R}^d}\int_{\mathbb{R}^d} K_{\delta}(x-y)H(x,y)d\mu(y)d\mu(x),
$$
where $H(x,y) = \tfrac{1}{2}[f(y)g(x) - g(y)f(x)]$.  The function $H$ is Lipschitz continuous on $\mathbb{R}^d\times\mathbb{R}^d$, and thus, as a consequence of the $\Lambda $-niceness of $\mu$, $K(x-y)H(x,y)\in L^1(\mu\times\mu)$.   We set
$$I_{\mu}(f,g)= \int_{\mathbb{R}^d}\int_{\mathbb{R}^d} K(x-y)H(x,y)d\mu(y)d\mu(x).
$$
Not only does the dominated convergence theorem yield that $I_{\mu, \delta}(f,g)\rightarrow I_{\mu}(f,g)$ as $\delta\rightarrow 0$, but we also have a simple quantitative estimate on the speed of convergence, namely that $|I_{\delta}(f,g)-I(f,g)|\leq C(f,g)\delta,$ where $C(f,g)$ depends on $f$ and $g$.  We may now define an operator $T_{\mu}$ from the space of compactly supported Lipschitz functions $f$ to its dual with respect to the pairing in $L^2(\mu)$ by
\begin{equation}\label{Ropdef}\langle T_{\mu}(f),g\rangle = I(f,g)
\end{equation}
for $f,g\in \Lip_0(\mathbb{R}^d).$

Assuming in addition that $\mu$ is $\Lambda $-good, we may define $I_{\mu, \delta}(f,g) = \langle T_{\mu, \delta}(f), g\rangle_{\mu}$ for any $f,g\in L^2(\mu)$.  Whence, the density of $\Lip_0(\mathbb{R}^d)$ in $L^2(\mu)$ allows us to extend the domain of the bilinear form $I(f,g)$ to all $f,g\in L^2(\mu)$, and furthermore $|I(f,g)|\leq \Lambda \|f\|_{L^2(\mu)}\|g\|_{L^2(\mu)}$.  Consequently, by the Riesz-Fisher theorem, we obtain a unique linear operator $T_{\mu}:L^2(\mu)\rightarrow L^2(\mu)$ with operator norm $\Lambda $ given by (\ref{Ropdef}) for $f,g\in L^2(\mu)$.

Should it happen that $\int_{\mathbb{R}^d}\int_{\mathbb{R}^d}|K(x-y)||f(y)||g(x)|d\mu(y)d\mu(x)<\infty$ for some $f,g\in L^2(\mu)$, then $$\langle T_{\mu}(f),g\rangle_{\mu} = \int_{\mathbb{R}^d}\int_{\mathbb{R}^d}K(x-y)f(y)g(x)d\mu(y)d\mu(x).$$
Indeed, in this case we have that $I_{\delta}(f,g)$ converges to $\int_{\mathbb{R}^d}\int_{\mathbb{R}^d}K(x-y)f(y)g(x)d\mu(y)d\mu(x)$ by the dominated convergence theorem.  So by the uniqueness of the operator $T_{\mu}$, we arrive at the desired formula.

For a $\Lambda $-good measure $\mu$, define
$T_{\mu}^{\delta}(f) = T_{\mu}(f) - T_{\mu, \delta}(f), \text{ for }f\in L^2(\mu).$
The operator $T_{\mu}^{\delta}:L^2(\mu)\rightarrow L^2(\mu)$ has operator norm $2\Lambda $.

Insofar as the kernel $K_{\delta}$ is antisymmetric, for any $B(x,r)\subset\mathbb{R}^d$,  $\int_{B(x,r)}T_{\mu, \delta}(\chi_{B(x,r)}) d\mu=0$.  Letting $\delta\rightarrow 0$ yields $\int_{B(x,r)}T_{\mu}(\chi_{B(x,r)}) d\mu=0$.  Consequently, for any $\delta>0$, $\int_{B(x,r)}T_{\mu}^{\delta}(\chi_{B(x,r)}) d\mu=0$ for any ball $B(x,r)$.

\subsection{The weak continuity of $T_{\mu}$}  Suppose that $\mu_k$ are $\Lambda $-good measures that converge weakly to a measure $\mu$.  By standard weak lower semi-continuity properties of the weak limit, we have that $\mu$ is $\Lambda $-nice.  In this section we show that $\mu$ is $\Lambda $-good.

Using an approximation result (such as the Stone-Weierstrass theorem, see page 7 of \cite{JN1}), it is not difficult to see that for any $f,g\in \Lip_0(\mathbb{R}^d)$, and $\delta>0$, \begin{equation}\label{deltawc}\lim_{k\rightarrow\infty} \langle T_{\mu_k, \delta}(f), g\rangle_{\mu_k} =   \langle T_{\mu, \delta}(f), g\rangle_{\mu}.\end{equation}
Recall that for any $\Lambda $-nice measure $\nu$, the convergence of $ \langle T_{\nu, \delta}(f), g\rangle_{\nu} $ to $ \langle T_{\nu}(f), g\rangle_{\nu} $ is uniform, and in particular depends only on $f,g$ and $\Lambda $.  Each $\mu_k$, as well as $\mu$,  is $\Lambda $-nice, so we have $\lim_{k\rightarrow\infty} \langle T_{\mu_k}(f), g\rangle_{\mu_k} =   \langle T_{\mu}(f), g\rangle_{\mu}$, for any $f,g\in \Lip_0(\mathbb{R}^d)$.

Referring again to (\ref{deltawc}), we employ the uniform $\Lambda $-goodness of the sequence $\mu_k$ to deduce that, for $f,g\in \Lip_0(\mathbb{R}^d)$,
$$|\langle T_{\mu, \delta}(f), g\rangle_{\mu}|\leq \Lambda \limsup_{k\rightarrow\infty}\|f\|_{L^2(\mu_k)}\|g\|_{L^2(\mu_k)}.
$$
The right hand side here is equal to $\Lambda \|f\|_{L^2(\mu)}\|g\|_{L^2(\mu)},$  since $|f|^2$ and $|g|^2$ are compactly supported continuous functions.
Considering that $\mu$ is $\Lambda $-nice,  $T_{\mu, \delta}:L^2(\mu)\rightarrow L^2_{\text{loc}}(\mathbb{R}^d)$ (see Section \ref{primer}).   Thus, the density of $\Lip_0(B(0,R))$ in $L^2(B(0,R))$ ensures that $\|T_{\mu, \delta}\|_{L^2( \mu)\rightarrow L^2(B(0,R),\,\mu)}\leq \Lambda $ for any $R>0$.  Using the monotone convergence theorem, we conclude that $\mu$ is $\Lambda $-good.

%Set $R_{\mu}:L^2(\mu)\rightarrow L^2(\mu)$.

%Set $R_{\mu}^{\delta}(f) = R_{\mu,\delta}(f) - R_{\mu}(f):L^2(\mu)\rightarrow L^2(\mu).$

\section{Riesz systems}
Let $\mu$ be a $\Lambda$-nice measure.  In this appendix we will prove Lemma \ref{systemlem}. % We shall prove a stronger result.  Let $A>10\sqrt{d}$.  For each $Q\in \mathcal{D}$, with $\mu(B(x_Q,A\ell(Q)))>0$, define
%\begin{equation}\begin{split}\nonumber\widetilde{\Psi}^{\mu}_{Q,A} = \Bigl\{\psi\in \Lip_0\bigl(B\bigl(x_Q,& A\ell(Q)\bigl)\bigl):\int_{\mathbb{R}^d}\psi d\mu=0, \text{ and }\\
%&\|\psi\|_{\Lip}< \frac{1}{\ell(Q)\sqrt{\mu(B(x_Q, 2A\ell(Q))}}\Bigl\}.\end{split}\end{equation}
%If $\mu(B(x_Q, A\ell(Q)))=0$, then set $\widetilde{\Psi}^{\mu}_{Q,A}=\{ 0 \}$.  We shall show that

\begin{lem}\label{systemlemap} $\Psi_Q^{\mu}$ $(Q\in \mathcal{D})$ is a $CA^{d+\tfrac{3s}{2}+2}$-Riesz family.
\end{lem}

%Since $\mu$ is $\Lambda$-nice, this readily implies Lemma \ref{systemlem}

\begin{proof}  For each $Q\in \mathcal{D}$, choose $\psi_Q\in {\Psi}_{A,Q}^{\mu}$.  Then
$$\|\psi_Q\|_{\infty}\leq\|\psi_{Q}\|_{\Lip}\cdot \diam(\supp(\mu))\leq \frac{CA}{\ell(Q)^{\tfrac{s}{2}}}.
$$
Thus,
$$\|\psi_Q\|_{L^1(\mu)} \leq CA\frac{\mu(B(x_Q, A\ell(Q))}{\ell(Q)^{\tfrac{s}{2}}}. %\text{ and } \|\psi_Q\|_{L^2(\mu)}\leq CA^{1+\tfrac{s}{2}}.
$$
Notice that if $Q', Q''\in \mathcal{D}$ with $\ell(Q')\leq \ell(Q'')$, then the oscillation of $\psi_{Q''}$ on $B(x_{Q'},A\ell(Q'))$ is bounded by $\displaystyle\tfrac{A\ell(Q')}{\ell(Q'')^{1+\tfrac{s}{2}}}$.  Thus
\begin{equation}\begin{split}\nonumber|\langle \psi_{Q'}, \psi_{Q''}\rangle_{\mu}| &\leq CA^2\frac{\ell(Q')}{\ell(Q'')} \frac{\mu(B(x_{Q'}, A\ell(Q')))}{\ell(Q')^{\tfrac{s}{2}}\ell(Q'')^{\tfrac{s}{2}}}\\&\leq CA^{2+\tfrac{s}{2}}\frac{\ell(Q')}{\ell(Q'')} \frac{\sqrt{\mu(B(x_{Q'}, A\ell(Q')))}}{\ell(Q'')^{\tfrac{s}{2}}} .\end{split}\end{equation}
%There are also the trivial estimates $|\langle \psi_{Q'}, \psi_{Q''}\rangle_{\mu}|\leq CA^2$ for any $Q',Q''\in \mathcal{D}$, and
Also, $|\langle \psi_{Q'}, \psi_{Q''}\rangle_{\mu}|=0$ if $B(z_{Q'}, A\ell(Q'))\cap B(z_{Q''}, A\ell(Q''))=\varnothing$.

For the remainder of this proof, all sums over cubes will be taken over the dyadic lattice $\mathcal{D}$, so we shall not write this explicitly.  Now, let $(a_Q)_{Q\in \mathcal{D}}\in \ell^2(\mathcal{D})$, then
$$\Bigl\|\sum_{Q}a_Q\psi_Q\Bigl\|_{\mu}^2 \leq 2\sum_{Q', Q'': \ell(Q')\leq \ell(Q'')}|a_{Q'}||a_{Q''}| |\langle \psi_{Q'}, \psi_{Q''}\rangle_{\mu}|
$$
%We first take care of the part of the sum where there cubes have equal side length.  Applying Cauchy's inequality, we have (after relabeling)
%$$\sum_{\substack{Q', Q'': \\\ell(Q')= \ell(Q'')}}|a_{Q'}||a_{Q''}| |\langle \psi_{Q'}, \psi_{Q''}\rangle_{\mu}|\leq \sum_{\substack{Q', Q'': \\\ell(Q')= \ell(Q'')}}|a_{Q'}|^2 |\langle \psi_{Q'}, \psi_{Q''}\rangle_{\mu}|.
%$$
%Thus, it suffices to show that, with $Q'\in \mathcal{D}$
%$$\sum_{\substack{Q'': \ell(Q')= \ell(Q'')\\ B(x_{Q'}, A\ell(Q'))\cap B(x_{Q''}, A\ell(Q''))\neq\varnothing}} |\langle \psi_{Q'}, \psi_{Q''}\rangle_{\mu}|,
%$$
%is bounded independently of $Q'$.  But this is seen easily as at most $CA^d$ cubes $Q''$ can participate in the sum.  For each such cube the inner product is bounded by $CA^2$.

%It remains to estimate the contribution by cubes $\ell(Q')<\ell(Q'')$.  Note that in this case, if $B(z_{Q'}, A\ell(Q'))\cap B(z_{Q''}, A\ell(Q''))\neq\varnothing$, then $B(z_{Q'}, 2A\ell(Q'))\subset B(z_{Q''}, 2A\ell(Q''))$.

Appealing to the bound on $\langle \psi_{Q'}, \psi_{Q''}\rangle_{\mu}$,  Cauchy's inequality yields that $|a_{Q'}||a_{Q''}| |\langle \psi_{Q'}, \psi_{Q''}\rangle_{\mu}|$ is bounded by
$$CA^{2+\tfrac{s}{2}}\Bigl[\frac{|a_{Q'}|^2}{2}  \frac{\ell(Q')}{\ell(Q'')}+ \frac{|a_{Q''}|^2}{2}\frac{\ell(Q')}{\ell(Q'')} \frac{\mu(B(x_{Q'}, A\ell(Q')))}{\ell(Q'')^s}\Bigl].
$$
Consequently, it suffices to estimate two sums:
$$I= \sum_{\substack{Q', Q'': \ell(Q')\leq \ell(Q'')\\ B(x_{Q'}, A\ell(Q'))\cap B(x_{Q''}, A\ell(Q''))\neq\varnothing}}\!\!\!\!\!\!\!|a_{Q'}|^2 \frac{\ell(Q')}{\ell(Q'')},$$ and
$$II= \sum_{\substack{Q', Q'': \ell(Q')\leq \ell(Q'')\\ B(x_{Q'}, A\ell(Q'))\cap B(x_{Q''}, A\ell(Q''))\neq\varnothing}}\!\!\!\!\!\!\!|a_{Q''}|^2 \frac{\ell(Q')}{\ell(Q'')} \frac{\mu(B(x_{Q'}, A\ell(Q')))}{\ell(Q'')^s}.$$
%There are two sums here, which we call $I$ and $II$.
%Since $\ell(Q')<\ell(Q'')$, and $B(z_{Q'}, \tfrac{A}{2}\ell(Q'))\cap B(z_{Q''}, \tfrac{A}{2}\ell(Q''))\neq\varnothing$, we have that $B(z_{Q'}, A\ell(Q'))\subset B(z_{Q''}, A\ell(Q''))$.
%Thus
Fix $Q'$ and $k\in \mathbb{Z}_+$.  There are at most $CA^d$ cubes $Q''$ with $\ell(Q'')=2^k\ell(Q')$ with $ B(x_{Q'}, A\ell(Q'))\cap B(x_{Q''}, A\ell(Q''))\neq\varnothing$. Thus
$$I\leq \sum_{Q'}|a_{Q'}|^2\!\!\!\!\!\!\!\!\!\!\!\!\!\! \sum_{\substack{Q'': \ell(Q')\leq \ell(Q'')\\ B(x_{Q'}, A\ell(Q'))\cap B(x_{Q''}, A\ell(Q''))\neq\varnothing}}\!\!\!\!\!\!\!\!\!\!\frac{\ell(Q')}{\ell(Q'')} \leq CA^d \sum_{Q'}|a_{Q'}|^2\sum_{k\in \mathbb{Z}_+}2^{-k}.
$$
For $II$, write
$$II = \sum_{Q''}|a_{Q''}|^2 \sum_{k\in \mathbb{Z}_+}2^{-k}\!\!\!\!\!\!\! \sum_{\substack{Q':\ell(Q') = 2^{-k}\ell(Q''))\\ B(x_{Q'}, A\ell(Q'))\cap B(x_{Q''}, A\ell(Q''))\neq\varnothing}} \frac{\mu(B(x_{Q'}, A\ell(Q')))}{\ell(Q'')^s}.
$$
With $k\in \mathbb{Z}_+$ fixed, the inner sum is at most
$$\frac{1}{\ell(Q'')^s}\!\!\!\int\limits_{B(x_{Q''}, 2A\ell(Q''))}\!\!\sum_{\substack{Q':\ell(Q') = 2^{-k}\ell(Q''))%\\ B(x_{Q'}, A\ell(Q'))\cap B(x_{Q''}, A\ell(Q''))\neq\varnothing}
}}%\!\!\!\!\!\!\!\!\!\!\!\!\!\!\!\!\!\!\!
\!\!\!\!\bigl[\chi_{B(x_{Q'}, A\ell(Q'))}(y)\bigl]d\mu(y).
$$
But any $y\in B(x_{Q''}, A\ell(Q''))$, the integrand is bounded by $CA^d$.  Thus
$$II\leq CA^d\sum_{Q''}|a_{Q''}|^2 \sum_{k\in \mathbb{Z}_+}2^{-k}\frac{\mu(B(x_{Q''},A\ell(Q'')))}{\ell(Q'')^s}\leq A^{d+s}\sum_{Q''}|a_{Q''}|^2.
$$
This completes the proof.
\end{proof}

\section{The Fourier transform and the fractional Laplacian}

In this appendix we establish Lemmas \ref{philem} and \ref{sharmon}.  Let $K(x)=\tfrac{x}{|x|^{s+1}}$ equal the $s$-Riesz kernel.  Thus, we may take $\alpha=1$ in the smoothness property (iii) of a CZO.  Throughout this section suppose that $\mu$ is a nontrivial $\Lambda $-good measure.   Fix the ball $B'$ and the function $\eta_{B'}$ as in the definition (\ref{R1Leb}).

\begin{lem}\label{Tballint}  Suppose that $A>[\dist(0, B')+\diam (B')]$.  Then
$$\int_{B(0,A)}|\R1_{\mu}(1)|dm_d \leq C(B')R^d\log(e+ A).
$$
\end{lem}

\begin{proof}
For $x\in B(0,A)$, we may write
\begin{equation}\begin{split}\nonumber\R1_{\mu}(1)(x)= &R(\chi_{B(0,2A)}\mu)(x) - \int_{B'}\eta_{B'} R_{\mu}(\chi_{B(0,2A)})d\mu \\
&+ \int_{\mathbb{R}^d\backslash B(0, 2A)}\int_{B'}\eta_{B'}(z)\bigl[K(x-y)-K(z-y) \bigl]d\mu(z)d\mu(y).
\end{split}\end{equation}
We shall estimate the integral of each term over $B(0,A)$ respectively.  First note that by Lemma \ref{locl1}, we have
$$
\int_{B(0,A)} |R(\chi_{B(0,2A)}\mu)|dm_d\leq CA^d.
$$
For the second term, we write
$$\int_{B'}\eta_{B'} R_{\mu}(\chi_{B(0,2A)}) = \int_{B'}\eta_{B'} R_{\mu}(\chi_{2B'})d\mu + \int_{B'}\eta_{B'} R_{\mu}(\chi_{B(0,2A)\backslash 2B'})d\mu.
$$
These integrals have a sum in absolute value no larger than $C(B')\log (e+A)$ (since $\mu$ is $\Lambda $-nice).
As for the remaining term defining $\R1_{\mu}(1)$, Lemma \ref{taildiff} applies to yield
$$
\int_{\mathbb{R}^d\backslash B(0, 2A)}\int_{B'}\eta_{B'}(z)\bigl|K(x-y)-K(z-y) \bigl|d\mu(z)d\mu(y)\leq \frac{C(B')}{A},
$$
for any $x\in B(0,A)$.
The lemma follows.
\end{proof}

Define $\mu_N = \chi_{B(0, N)}\mu$.

\begin{lem}\label{RL1est}
Set $A>0$. Then for all sufficiently large $N>0$,
$$\|\R1_{\mu}(1) - \R1_{\mu_N}(1)\|_{L^{\infty}(B(0,A))}\leq \frac{C(B')}{N}.
$$
\end{lem}

\begin{proof} The proof is essentially the same as that of Lemma \ref{1closetophi}.  Fix $x\in B(0,A)$, and choose $B=B(0, \tfrac{N}{2})$ in (\ref{R1Leb}) (permitted provided that $N$ is large enough).  Then
$$\R1_{\mu}(1)(x)-\R1_{\mu_N}(1)(x) = \int_{\mathbb{R}^d\backslash B(0, N)}\int_{B'}\eta_{B'}\bigl[K(x-y)-K(z-y)] d\mu(z)d\mu(y),
$$
provided that $N$ is sufficiently large.  But this quantity is bounded in absolute value by $\tfrac{C(B')}{N}$.
\end{proof}

Let us now prove Lemma \ref{philem}

\begin{proof}[Proof of Lemma \ref{philem}]  We may assume that $x=0$, and $r=1$.

First let $\sigma$ be a finite measure.  Fix a smooth compactly supported nonnegative bump function $f$ that is identically $1$ on $B(0,1)$.

Let $g(x) = \mathcal{F}^{-1}( b\xi|\xi|^{d-1-s}\hat{f}(\xi))(x)$, where $b\in \mathbb{C}\backslash\{0\}$ is chosen so that $R^{*}(g)=f$.  Note that $g$ has $m_d$-mean zero in each component, since the vector field $\xi|\xi|^{d-1-s}$ vanishes when $\xi=0$.  An elementary decomposition on the Fourier side yields the estimate
$$|g(x)|\leq \frac{C}{(1+|x|)^{2d-s}}.
$$
Therefore $g\in L^1(m_d)\cap L^2(m_d)$, and by the construction of $g$ we have
$$\int_{\mathbf{R}^d} fd\mu = \int_{\mathbf{R}^d} R^{*}(g m_d) d\sigma = \int_{\mathbf{R}^d}R(\sigma)\cdot g dm_d.
$$
Using the mean zero property of $g$ and the estimate on its absolute value,
$$\sigma(B(0,1))\leq C\int_{\mathbb{R}^d} \frac{|R(\sigma)(y)-\widetilde{\Gamma}|}{(1+|y|)^{2d-s}}dm_d(y).
$$
for any $\widetilde{\Gamma}\in \mathbb{R}^d$.  Replacing $\sigma$ by $\mu_N$ with $N>1$, and $\widetilde{\Gamma}= \Gamma - \int_{B'}\eta_{B'} \R1_{\mu_N}(1)d\mu$, yields
$$\mu(B(0,1))\leq C\int_{\mathbb{R}^d} \frac{|\R1_{\mu_N}(1)(y)-\Gamma|}{(1+|y|)^{2d-s}}dm_d(y).
$$
It remains to pass to the limit as $N\rightarrow \infty$.

Fix $\eps>0$.  For any $A>0$, a standard application of Fubini's theorem yields that
$$\int_{\mathbb{R}^d\backslash B(0,A)} \frac{|\R1_{\mu_N}(1)(y)|+|\R1_{\mu}(1)(y)|}{(1+|y|)^{2d-s}}dm_d(y)$$ is bounded by a constant multiple of  $$\int_A^{\infty}\frac{1}{T^{2d-s}}\int_{B(0,T)} |\R1_{\mu_N}(1)(y)|+|\R1_{\mu}(1)(y)|dm_d(y) \frac{dT}{T}.
$$
Appealing to Lemma \ref{Tballint}, we estimate this quantity by $C(B')\int_{A}^{\infty}\tfrac{\log(e+ T)}{T^{d-s}}\tfrac{dT}{T}$, which is smaller than $\eps$ if $A$ is chosen sufficiently large.
But then provided that $N$ is sufficiently large, Lemma \ref{RL1est} yields that
$$\int_{B(0,A)}|\R1_{\mu_N}(1)(y) - \R1_{\mu}(1)(y)|dm_d(y) < \eps.
$$
The desired estimate follows.
\end{proof}

We now turn our attention to Lemma \ref{sharmon}.

\begin{lem}\label{secondder}  If $\dist(0, \supp(\mu))\geq 1$, then
$$|D^2(\widetilde{R}_{\mu}(1))(z)|\leq C \text{ for any }z\in B(0, \tfrac{1}{2}).
$$
\end{lem}

\begin{proof}
Since $\tfrac{1}{|x-\cdot|^{q}}\in L^1(\mu)$ for any $x\in B(0,\tfrac{1}{2})$ and $q>s$, with $L^1(\mu)$ norm depending only on $\Lambda $, $s$, and $q$, the desired estimate readily follows from a simple justification of differentiation under the integral.
\end{proof}

\begin{proof}[Proof of Lemma \ref{sharmon}]  By an appropriate resealing, we may assume that $B(0,1)\cap\supp(\mu)=\varnothing$.

We shall need a well known formula for the fractional Laplacian (see, for example \cite{Lan},\cite{ENV}): if $g\in C^{\infty}_{0}(\mathbb{R}^d)$, then
$$
\nabla g(x) = b P.V. \int_{\mathbb{R}^d}\frac{R(gm_d)(x)- R(gm_d)(y)}{|x-y|^{2d+1-s}}dm_d(y),
$$
where $b\in \mathbb{R}\backslash \{0\}$.

Suppose that $\psi\in C^{\infty}_0(B(0,1))$ satisfies $\int_{\mathbb{R}^d} \psi dm_d=1$.  For $\rho>0$ and $N>0$, set $\psi_{\rho}=\rho^{-n}\psi(\rho\,\cdot\,)$, $u^{(N)}_{\rho} =  \psi_{\rho}*\R1_{\mu_N}(1)$,   $u^{(N)}= \R1_{\mu_N}(1)$, and $u = \R1_{\mu}(1)$.  If $\rho<\tfrac{1}{2}$, then $\psi_{\rho}*\mu_N(0)=0$, so
$$0 = P. V. \int_{\mathbb{R}^d}\frac{u_{\rho}^{(N)}(0) - u_{\rho}^{(N)}(y)}{|y|^{2d+1-s}}dm_d(y),
$$
for any $N>0$.

Let $\eps>0$.  Due to the second derivative estimate (Lemma \ref{secondder}), if $\rho<\tfrac{1}{4}$ and $r<\tfrac{1}{4}$, then for any $N>0$,
\begin{equation}\begin{split}\nonumber \Bigl|P. V. \int_{B(0,r)}\frac{u^{(N)}_{\rho}(0) - u^{(N)}_{\rho}(y)}{|y|^{2d+1-s}}&dm_d(y)\Bigl|+ \Bigl|P. V. \int_{B(0,r)}\frac{u(0) - u(y)}{|y|^{2d+1-s}}dm_d(y)\Bigl|\\
&\leq C\int_0^r\frac{t^2}{t^{2d+1-s}}t^d\frac{dt}{t}\leq Cr^{1+s-d},
\end{split}\end{equation}
which is smaller than $\eps$ if $r$ is chosen small enough.

Next, as in the proof of Lemma \ref{philem} above, we apply Lemma \ref{Tballint} to find a radius $A>1$ such that
$$\int_{\mathbb{R}^d\backslash B(0,A)} \frac{|u_{\rho}^{(N)}(y)|+|u_{\rho}^{(N)}(0)|+ |u(y)|+ |u(0)|}{|y|^{2d+1-s}}dm_d(y)<\eps.
$$
for any $N>0$.  On the other hand, from Lemma \ref{RL1est}, we infer that for all sufficiently large $N$, we have
$$\int_{B(0,A)}|u^{(N)}(y) - u(y)|dm_d(y) \leq \eps r^{2d+1-s}.
$$
Of course also have that $|\R1_{\mu_{N}}(0)-\R1_{\mu}(1)(0)|\leq \tfrac{C(B')}{N} \leq \eps r^{2d+1-s}$ as long as $N$ is reasonably large.

With $A$ and $R$ now fixed,
$$\lim_{\rho\rightarrow 0}\int_{B(0,A)} |u^{(N)}_{\rho}(y) - u^{(N)}(y)|dm_d(y)= 0,$$
and $u_{\rho}^{(N)}(0) \rightarrow u^{(N)}(0)$ as $\rho\rightarrow 0$.  Thus, the triangle inequality yields
$$\Bigl|P. V. \int_{\mathbb{R}^d}\frac{u(0) - u(y)}{|y|^{2d+1-s}}dm_d(y)\Bigl|\leq 3\eps,
$$
as required.
\end{proof}

 \end{document}